\documentclass[a4paper,11pt]{amsart}

\usepackage[T1]{fontenc}
\usepackage{microtype}

\usepackage{amsmath,amsthm,amssymb}
\usepackage{mathtools}
\usepackage[all]{xy}

\usepackage[foot]{amsaddr}

\usepackage[textheight=24cm,textwidth=15cm]{geometry}

\usepackage[applemac]{inputenc}

\usepackage[final,pdftitle={Uniform boundedness principles for Sobolev maps into manifolds},pdfauthor={Antonin Monteil and Jean Van Schaftingen}]{hyperref}

\usepackage[abbrev,backrefs]{amsrefs}

\usepackage{paralist}

\usepackage{constants}

\usepackage{todonotes}
\usepackage[nameinlink,capitalize]{cleveref}

\usepackage{stix}

\theoremstyle{plain}
\newtheorem{theorem}{Theorem}[section]

\newtheorem*{prop*}{Proposition}
\newtheorem{lem}[theorem]{Lemma}

\newtheorem{claim}{Claim}

\theoremstyle{definition}
\newtheorem{definition}[theorem]{Definition}

\theoremstyle{remark}

\newtheorem{rmk}[theorem]{Remark}

\numberwithin{equation}{section}

\newcommand{\diff}{\,\mathrm{d}}
\newcommand{\R}{\mathbb{R}}
\newcommand{\N}{\mathbb{N}}

\newcommand{\st}{\;:\;}
\newcommand{\floor}[1]{\lfloor #1 \rfloor}
\newcommand{\abs}[1]{\lvert#1 \rvert}
\newcommand{\Bigabs}[1]{\Bigl\lvert#1 \Bigr\rvert}
\newcommand{\compose}{\,\circ\,}

\title{Uniform boundedness principles for Sobolev maps into manifolds}

\author{Antonin Monteil}
\author{Jean Van Schaftingen}
\address{Universit\'e catholique de Louvain\\ Institut de Recherche en Math\'ematique et Physique\\ Chemin du Cyclotron 2 bte L7.01.01\\ 1348 Louvain-la-Neuve\\ Belgium}
\email{Antonin.Monteil@uclouvain.be, Jean.VanSchaftingen@uclouvain.be}

\thanks{A. Monteil is a postdoctoral researcher (charg\'e de recherches) by the Fonds de la Recherche Scientifique--FNRS; J. Van Schaftingen was supported by the Mandat d'Impulsion Scientifique F.4523.17, ``Topological singularities of Sobolev maps'' of the Fonds de la Recherche Scientifique--FNRS}

\allowdisplaybreaks
\begin{document}
\begin{abstract}
Given a connected Riemannian manifold $\mathcal{N}$, an \(m\)--dimensional Riemannian manifold $\mathcal{M}$ which is either compact or the Euclidean space, $p\in [1, +\infty)$ and $s\in (0,1]$, we establish, for the problems of  surjectivity of the trace, of weak-bounded approximation, of lifting and of superposition, that qualitative properties satisfied by every map in a nonlinear Sobolev space $W^{s,p}(\mathcal{M}, \mathcal{N})$ imply  corresponding uniform quantitative bounds. This result is a nonlinear counterpart of the classical Banach--Steinhaus uniform boundedness principle in linear Banach spaces.
\end{abstract}



\maketitle


\section{Introduction}

When \(s \in (0,1)\) and \(p\in [1,+\infty)\), the Sobolev space $W^{s,p}(\mathcal{M}, \mathcal{N})$ of maps between the Riemannian manifolds \(\mathcal{M}\) and \(\mathcal{N}\) can be defined as 
\[
 W^{s, p} (\mathcal{M}, \mathcal{N})
 = \bigl\{ u : \mathcal{M} \to \mathcal{N} \text{ is measurable and } \mathcal{E}_{s, p} (u,\mathcal{M}) < + \infty \bigr\},
\]
where \(\mathcal{E}_{s, p}\) is the \emph{Gagliardo energy} for fractional Sobolev maps defined for a measurable map \(u : \mathcal{M} \to \mathcal{N}\) as
\begin{equation}
\label{eqDefGagliardo}
\mathcal{E}_{s,p}(u,\mathcal{M})=\int_\mathcal{M}\int_\mathcal{M} \frac{d_{\mathcal{N}}\big(u(x),u(y)\big)^p}{d_\mathcal{M}(x,y)^{m+s p}}\diff x\diff y,
\end{equation}
with \(d_\mathcal{N}\) and \(d_\mathcal{M}\) being the geodesic distances induced by the Riemannian metrics of the manifolds \(\mathcal{N}\) and \(\mathcal{M}\) and \(m = \dim \mathcal{M}\). When the manifold \(\mathcal{N}\) is embedded into a Euclidean space \(\R^\nu\) by a bi-Lipschitz embedding and \(\mathcal{N}\) is identified to this embedding's image, we have \(W^{s, p} (\mathcal{M}, \mathcal{N}) = \{ u \in W^{s, p}(\mathcal{M},\R^\nu) : u \in \mathcal{N} \text{ almost everywhere in } \mathcal{M}\}\) and the corresponding energies are comparable.

When \(s = 1\) we can assume by the Nash embedding theorem \cite{Nash1956} that the manifold \(\mathcal{N}\) is isometrically embedded into \(\R^\nu\), and we can define 
\[
 W^{1,p}(\mathcal{M},\mathcal{N})=\bigl\{u\in W^{1,p}(\mathcal{M},\R^\nu)\st u(x)\in\mathcal{N}\text{ for almost every }x\in\mathcal{M}\bigr\}
\]
and 
\[
 \mathcal{E}_{1,p}(u,\mathcal{M})=
  \int_\mathcal{M} |D u|^p.
\]
The space, the energy and the topology on this space are independent of the embedding and can be defined intrinsically \cite{ConventVanSchafingen2016}.

\subsection{Extension of traces} We first consider relationships between a qualitative and quantitative properties for the problem of \emph{surjectivity of the trace}. In the setting of \emph{linear Sobolev spaces}, given \(s \in (0, 1)\), \(p \in (1, +\infty)\) and a manifold \(\mathcal{M}\) which is either compact or the Euclidean space, the classical trace theory states that the restriction of continuous functions in \(C(\mathcal{M} \times [0, + \infty), \R)\)
has a linear continuous extension to the trace operator \(\operatorname{tr} : W^{s+{1}/{p}, p} (\mathcal{M} \times (0, + \infty),\R) \to W^{s, p} (\mathcal{M}, \R)\)
and that the latter trace operator is surjective \citelist{\cite{Triebel1983}*{Theorem 2.7.2}\cite{AdamsFournier2003}*{Theorem 7.39}\cite{diBenedetto2016}*{Chapter 10}}. By the proof of the surjectivity or by a straightforward application of Banach's open mapping theorem (see for example \cite{Brezis2011}*{theorem 2.6}), which can be deduced from the Banach--Steinhaus uniform boundedness principle, every function \(u \in W^{s, p} (\mathcal{M},\R)\) can be written as \(u = \operatorname{tr} U\), with a function \(U \in W^{s + {1}/{p}, p} (\mathcal{M} \times (0, +\infty),\R)\) whose norm is controlled by the norm of the function \(u\). 
When dealing with nonlinear Sobolev spaces \(W^{s, p} (\mathcal{M}, \mathcal{N})\) into a compact Riemannian manifold \(\mathcal{N}\), the trace operator remains a well-defined continuous operator. 
The question of its surjectivity is more delicate: if \(s = 1 - \frac{1}{p}\), \(sp \le m\) and if \(\pi_1 (\mathcal{N}) \simeq \dotsb \simeq \pi_{\floor{p} - 1} (\mathcal{N})\simeq \{0\}\) --- that is for every \(j \in \N\) such that \(j \le p - 1\), every continuous map \(f\) from the \(j\)--dimensional sphere into \(\mathcal{N}\) has a continuous extension from the \((j + 1)\)--dimensional ball to \(\mathcal{N}\) ---, then the trace operator is surjective \cite{HardtLin1987}*{Theorem 6.2}.
This topological condition is almost necessary: if the trace is surjective, then \(\pi_1 (\mathcal{N})\) is finite and \(\pi_2 (\mathcal{N}) \simeq \dotsb \simeq \pi_{\floor{p} - 1} (\mathcal{N}) \simeq \{0\}\) \cite{Bethuel2014Extension} (see also \cite{BethuelDemengel1995}).

In order to study quantitatively the problem, we introduce the \emph{extension energy},
defined for every \(r \in (0, 1]\) and \(q \in [1, +\infty)\) such that \(rq > 1\), for every manifold \(\mathcal{M}\) and every measurable map \(u:\mathcal{M}\to \mathcal{N}\) by 
\[
\mathcal{E}^{\mathrm{ext}}_{r,q}(u,\mathcal{M})
=\inf\,\Bigl\{\mathcal{E}_{r,q}(U,\mathcal{M}\times\R_+)
 \st U\in W^{r, q}(\mathcal{M}\times\R_+,\mathcal{N}) \text{ and } \operatorname{tr} U = u \Bigr\} \in [0, + \infty]
\]
(The condition \(rq > 1\) guarantees that the trace is well-defined.).
In particular, the surjectivity of the trace operator can be reformulated by stating that if \(\mathcal{E}_{s, p} (u, \mathcal{M}) < + \infty\) then \(\mathcal{E}^{\mathrm{ext}}_{s+1/p,p}(u,\mathcal{M}) < + \infty\).

Our first nonlinear uniform boundedness principle states that the surjectivity of the trace implies a linear uniform bound: 

\begin{theorem}\label{extensionthm}
Let \(s, r \in (0, 1]\), \(p, q \in [1, +\infty)\), \(m \in \N_*\), \(\mathcal{M}\) be a Euclidean space or a compact Riemannian manifold of dimension \(m\) and \(\mathcal{N}\) be a connected Riemannian manifold which is compact if either \(sp > m\) or \(s = p = m = 1\). If \(sp = rq - 1\) and if every map in \(W^{s, p} (\mathcal{M}, \mathcal{N})\) is the trace of some map in \(W^{r, q} (\mathcal{M} \times (0, + \infty), \mathcal{N})\), then there exists a constant $C>0$ such that for each measurable function 
$u: \mathbb{B}^m \to \mathcal{N}$ with either \(sp <m\) or \(\mathcal{E}_{s, p} (u, \mathbb{B}^m) \le 1/C\), then
\[
  \mathcal{E}^{\mathrm{ext}}_{r ,q}(u,\mathbb{B}^m)\leq C \, \mathcal{E}_{s,p}(u,\mathbb{B}^m),
\]
where $\mathbb{B}^m$ stands for the unit ball in $\R^m$.
\end{theorem}

When \(s = 1 - \frac{1}{p}\), \(r = 1\), \(p = q\), the manifold \(\mathcal{N}\) is compact and \(\pi_1 (\mathcal{N}) \simeq \dotsb \simeq \pi_{\floor{p} - 1} (\mathcal{N})\simeq \{0\}\), the estimate of \cref{extensionthm} was already known as a byproduct of the proof of the surjectivity of the trace by Hardt and Lin \cite{HardtLin1987}*{proof of Theorem 6.2}; some flavour of \cref{extensionthm} is present in Bethuel's counterexample \cite{Bethuel2014Extension}. \Cref{extensionthm} shows that these linear bounds are an essential feature for this class of problems. 

In the case where the trace operator is not surjective, since \(W^{1, q - 1} ( \mathcal{M}, \mathcal{N}) \subset W^{1 - 1/q, q} ( \mathcal{M}, \mathcal{N})\), one can still wonder whether any map in this smaller space is the trace of a map in \(W^{1, q} (\mathcal{M}\times (0,+\infty), \mathcal{N})\). \Cref{extensionthm} shows that this would still imply a weaker uniform estimate.

The smallness restriction on the energy when \(sp \ge m\) is related in the proof to scaling properties of Sobolev energies and ensures that moving a map to smaller scales decreases the Sobolev energy. Moreover, extension results for \(sp \ge m\) are proved by patching a nearest point retraction of an extension together with a smooth extension of a smooth map \cite{BethuelDemengel1995}*{Theorems 1 and 2} for which there does not seem to be an immediate linear bound; when \(sp > m\) a compactness argument leads to a nonlinear estimate of the norm of the extension by the norm of the trace which has no reason to be linear \cite{PetracheVanSchaftingen}*{Theorem 4}. 
When \(s = 1-1/p\) and \(\mathcal{N}\) is a compact Riemannian manifold such that either \(\pi_1 (\mathcal{N})\) is infinite or \(\pi_{j} (\mathcal{N})\not\simeq\{0\}\) for some \(j \le p - 1\), there exists a sequence \((u_n)_{n \in \mathbb{N}}\) in \(W^{1 - 1/p, p}(\mathbb{B}^m, \mathcal{N})\) such that \cite{Bethuel2014Extension}*{(1.36)}
\begin{equation}
\label{eqBethuelGrowth}
  \liminf_{n \to \infty}   \frac{\mathcal{E}^\mathrm{ext}_{1, p} (u_n,\mathbb{B}^m)}{\mathcal{E}^{p/(p - 1)}_{1-1/p, p} (u_n,\mathbb{B}^m)} > 0\quad\text{and}\quad \lim_{n\to\infty}\mathcal{E}_{1-1/p,p}(u_n,\mathbb{B}^m)=+\infty,
\end{equation}
ruling thus out the extension of the estimate of \cref{extensionthm} when \(sp \ge m\) for large Sobolev energies.

In the limit case \(s \to 1\) and \(p \to +\infty\), the problem of quantitative bounds has some analogy with the construction of controlled Lipschitz homotopies to constant maps \cite{Gromov1999}, whose answer depends on the finiteness of the first homotopy groups of the target manifold \(\mathcal{N}\) \cite{FerryWeinberger2013}.

\subsection{Weak-bounded approximation}
Smooth functions are known to be dense in the Sobolev space \(W^{s, p} (\mathcal{M}, \R)\) with respect to the strong topology induced by the norm.
The \emph{strong approximation problem} asks whether any Sobolev map in \(W^{s, p} (\mathcal{M}, \mathcal{N})\) can be approximated in the strong topology by smooth maps in \(C^\infty (\mathcal{M}, \mathcal{N})\). When \(sp \ge m\),  and \(\mathcal{N}\) is compact, the answer is positive and related to the fact that maps in \(W^{s, p} (\mathcal{M}, \mathcal{N})\) are continuous when \(sp > m\) and have vanishing mean oscillation (VMO) when \(sp = m\) \cite{SchoenUhlenbeck1983}*{\S 4}. When \(sp < m\), the answer is delicate and depends on the homotopy type of the pair \((\mathcal{M}, \mathcal{N})\) \citelist{\cite{bethuel1991approximation}\cite{hang2003topology}\cite{BrezisMironescu2015}}. In the particular case where the domain \(\mathcal{M}\) is a ball, a necessary and sufficient condition for strong density is that \(\pi_{\floor{sp}} \simeq \{0\}\), that is, every continuous map \(f \in C (\mathbb{S}^{\floor{sp}}, \mathcal{N})\) is the restriction of some continuous map \(F \in C (\mathbb{B}^{\floor{sp + 1}}, \mathcal{N})\).

When strong density of smooth maps does not hold, one can still wonder whether a map \(u \in W^{s, p} (\mathcal{M}, \mathcal{N})\) has a \emph{weak-bounded approximation}, that is, whether there exists a sequence \((u_i)_{i \in \N}\) in \(C^\infty (\mathcal{M}, \mathcal{N})\) that converges almost everywhere to \(u\) and for which the sequence of Sobolev energies \((\mathcal{E}_{s, p} (u_i))_{i \in \N}\) remains bounded. 
When \(p > 1\) and the manifold \(\mathcal{N}\) is compact, the weak-bounded convergence  is equivalent to the weak convergence induced by the embedding of \(\mathcal{N}\) in the Euclidean space \(\R^\nu\).
In the nonintegral case \(sp \not \in \N\), a map \(u \in W^{s, p} (\mathcal{M}, \mathcal{N})\) has a weak-bounded approximation if and only if it has a strong approximation \cite{bethuel1991approximation}*{Theorem 3 bis}. The remaining \emph{interesting case} is thus the integral case \(sp \in \mathbb{N}\).

Hang and Lin have given a necessary condition on the homotopy type of the pair \((\mathcal{M}, \mathcal{N})\) so that each map in \(W^{1, p} (\mathcal{M}, \mathcal{N})\) has a weak-bounded approximation \cite{hang2003topology}*{Theorem 7.1}.
Every map in \(W^{1, p} (\mathcal{M},\mathcal{N})\) is known to have a weak-bounded approximation when \(\mathcal{N} = \mathbb{S}^p\) \citelist{\cite{BrezisCoronLieb1986}\cite{Bethuel1990}\cite{bethuel1991approximation}*{Theorem 6}} or \(\pi_1 (\mathcal{N}) \simeq \dotsb \simeq \pi_{p - 1} (\mathcal{N}) \simeq \{0\}\) \cite{Hajlasz1994} (see also \citelist{\cite{Hang:2003}*{Proposition 8.3}\cite{BousquetPonceVanSchaftingen2013}*{Theorem 1.7}}), when \(p = 2\) \cite{PakzadRiviere2003} and \(p = 1\) \citelist{\cite{PakzadRiviere2003}\cite{Pakzad2003}}. 
On the other hand, when \(m \ge 4\) there exists a map \(u \in W^{1, 3} (\mathcal{M}, \mathbb{S}^2)\) that does not have any weak-bounded approximation \cite{Bethuel2014Weak}.
In the fractional case, it is known that any map in \(W^{1/2, 2} (\mathbb{S}^2, \mathbb{S}^1)\) has a weak-bounded approximation \cite{Riviere2000}.

Following Bethuel, Brezis and Coron \citelist{\cite{BethuelBrezisCoron1990}\cite{Brezis1991}\cite{Brezis1993}}, we define the \emph{relaxed energy} for every manifold \(\mathcal{M}\) and every measurable map \(u:\mathcal{M}\to \mathcal{N}\) by 
\begin{multline*}
\mathcal{E}^{\mathrm{rel}}_{s,p}(u,\mathcal{M}):=\inf\, \Bigl\{\liminf\limits_{n\to\infty}\mathcal{E}_{s,p}(u_n,\mathcal{M})\st  \text{for each \(n \in \N\), } u_n \in \mathcal{C}^\infty (\mathcal{M},\mathcal{N}) \\ \text{ and }
u_n \to u\text{ almost everywhere as } n \to \infty \Bigr\}.
\end{multline*}
A map \(u \in W^{s, p} (\mathcal{M},\mathcal{N})\) has a weak-bounded approximation in \(W^{s, p} (\mathcal{M},\mathcal{N})\) if and only if \(\mathcal{E}^{\mathrm{rel}}_{s,p}(u,\mathcal{M}) < +\infty\).

\begin{theorem}\label{approximationthm}
Let $s, r\in (0, 1]$, $p, q \in [1, +\infty)$, $m \in \N_*$, \(\mathcal{M}\) be a Euclidean space or a compact Riemannian manifold of dimension \(m\) and let \(\mathcal{N}\) be a connected Riemannian manifold.
If $sp = rq <m$ and if every map \(u \in W^{s, p} (\mathcal{M},\mathcal{N})\) has a weak-bounded approximation in \(W^{r, q} (\mathcal{M},\mathcal{N})\), then there exists a constant $C>0$ such that for each measurable function $u:\mathbb{B}^m \to \mathcal{N}$, one has
\[
  \mathcal{E}^{\mathrm{rel}}_{r,q}(u,\mathbb{B}^m)
  \leq C \,\mathcal{E}_{s,p}(u,\mathbb{B}^m).
\]
\end{theorem}

\Cref{approximationthm} extends trivially to the case where \(sp \ge m\) and the target manifold \(\mathcal{N}\) is compact, since every map has then a strong approximation and thus for every \(u\in W^{s, p} (\mathbb{B}^m,\mathcal{N})\),
\(\mathcal{E}^\mathrm{rel}_{s, p} (u,\mathbb{B}^m) = \mathcal{E}_{s, p} (u,\mathbb{B}^m)\).
In the situation where \(sp = m\) and the manifold \(\mathcal{N}\) is not compact, either \(\mathcal{N}\) is sufficiently nondegenerate at infinity to satisfy the trimming property that implies that every map has then a strong approximation \cite{BousquetPonceVanSchaftingen2017} and therefore the relaxed energy coincides with the Sobolev energy, or the trimming property fails and there exists a map that has no weak-bounded approximation \cite{BousquetPonceVanSchaftingen2018}.

\Cref{approximationthm} also implies that if every map in \(W^{1, p} (\mathcal{M}, \mathcal{N})\) has a weak-bounded approximation in the larger space \(W^{s, {p}/{s}}  (\mathcal{M}, \mathcal{N})\), with \(s\in (0,1)\), then a similar uniform boundedness principle has to hold.

When \(s = r = 1\) and \(p=q\), \cref{approximationthm} is due to Hang and Lin \cite{Hang:2003}*{Theorem 9.6}; \cref{approximationthm} is also present in the final step of the construction of the counterexample to the weak-bounded approximation in \(W^{1, 3} (\mathcal{M}, \mathbb{S}^2)\) \cite{Bethuel2014Weak}.

\subsection{Lifting}
Another situation in which Sobolev maps enjoy a uniform bound principles is the lifting problem.
Given a manifold \(\mathcal{F}\) and a Lipschitz map \(\pi : \mathcal{F} \to \mathcal{N}\), it can be checked immediately that if \(\varphi \in W^{s, p} (\mathcal{M}, \mathcal{F})\), then \(\pi \compose \varphi \in W^{s, p} (\mathcal{M}, \mathcal{N})\). The lifting problem asks whether every map \(u \in W^{s, p} (\mathcal{M}, \mathcal{N})\) can be lifted to a map \(\varphi \in W^{s, p} (\mathcal{M}, \mathcal{F})\) such that \(\pi \compose \varphi = u\) on \(\mathcal{M}\), that is, there exists \(\varphi\) such that the diagram
\[
\xymatrix{
\mathcal{M}\ar[r]^u \ar[d]_\varphi&\mathcal{N}\\
\mathcal{F} \ar[ru]_\pi
}
\]
commutes. In other words, we wonder whether the composition operator \(\varphi\in W^{s, p} (\mathcal{M}, \mathcal{F})\to \pi \compose \varphi \in W^{s, p} (\mathcal{M}, \mathcal{N})\) is surjective.

This lifting problem has been the object of a detailed study when \(\mathcal{N}\) is the unit circle \(\mathbb{S}^1\) and \(\pi : \R\to\mathbb{S}^1\) is its \emph{universal covering}, defined by \(\pi (t) = (\cos t, \sin t)\) for every \(t \in \R\).
In this case, when the manifold \(\mathcal{M}\) is simply-connected, every map in \(W^{s, p} (\mathcal{M}, \mathbb{S}^1)\) admits a lifting if and only if either \(s = 1\) and \(p \ge 2\), or \(s < 1\) and \(sp < 1\), or \(s < 1\) and \(sp \ge m\) \cite{bourgain2000lifting}.
Similar results hold for the universal covering \(\pi : \mathcal{F} \to \mathcal{N}\) when the fundamental group \(\pi_1 (\mathcal{N})\) is infinite \cite{bethuel2007some};
when \(\pi_1(\mathcal{N})\) is a nontrivial finite group, it is not yet known whether the condition \(sp \not \in [1, m)\) is necessary when \(s < 1\).
These results apply to the case of the universal covering of the projective space \(\mathbb{R}P^m\) by the sphere \(\mathbb{S}^m\) when \(m \ge 2\) \citelist{\cite{Mucci2012}\cite{BallZarnescu2011}}.

Another lifting problem that has been studied is the \emph{lifting problem for fibrations}. For the Hopf fibration \(\pi : \mathbb{S}^3 \to \mathbb{S}^2\), in contrast with the universal covering, some gauge invariance property shows that the existence of one lifting implies the presence of a continuum of liftings and a lifting is known to exist when \(s = 1\) and  \(1 \le p < 2 \le m\) or \(p \ge m \ge 3\) or \(p > m = 2\) \cite{bethuel2007some}, and known to be impossible for some map if \(2 \le p < m\) \citelist{\cite{bethuel2007some}\cite{Bethuel2014Weak}}.

To quantify the lifting of a Sobolev map we define the \emph{lifting energy} of a map \(u:\mathcal{M}\to\mathcal	{N}\) by  
\[
\mathcal{E}^\mathrm{lift}_{s,p}(u,\mathcal{M}):=\inf\,\left\{\mathcal{E}_{s,p}(\varphi,\mathcal{M})\st \varphi : \mathcal{M} \to \mathcal{F} \text{ is measurable and } \pi \compose\varphi = u \right\}.
\]
When \(s = 1\) and \(p \ge 1\), the lifting \(W^{1, p} (\mathcal{M}, \mathbb{S}^1)\) preserves the Sobolev energy; when \(s < 1\) and \(sp < 1\), the existing bounds on liftings of maps in \(W^{s, p} (\mathcal{M}, \mathbb{S}^1)\) are linear \citelist{\cite{bourgain2000lifting}} (see also \cite{MironescuMolnar2015}) and suggest the following uniform boundedness principle:

\begin{theorem}\label{liftingthm}Let $s,r\in (0, 1]$, $p,q\in [1, +\infty)$, $m \in \N_*$, \(\mathcal{M}\) be a Euclidean space or a compact Riemannian manifold of dimension \(m\), \(\mathcal{N}\) and \(\mathcal{F}\) be Riemannian manifolds with \(\mathcal{N}\) connected and, if either \(sp > m\) or \(s = p = m = 1\), compact, and \(\pi : \mathcal{F} \to \mathcal{N}\). If \(rq = sp\) and if for every map in $W^{s,p}(\mathcal{M},\mathcal{N})$ there exists \(\varphi \in W^{r, q} (\mathcal{M}, \mathcal{F})\) such that \(\pi \compose \varphi = u\), then there exists a constant $C>0$ such that for each measurable function \(u:\mathbb{B}^m\to \mathcal{N}\), if either \(sp < m\) or  \(\mathcal{E}_{s,p}(u,\mathbb{B}^m) \le 1/C\),
\[
  \mathcal{E}^\mathrm{lift}_{r,q}(u,\mathbb{B}^m)\leq C\, \mathcal{E}_{s,p}(u,\mathbb{B}^m).
\]
\end{theorem}

The restriction \(sp <m\) for avoiding the smallness condition comes again from the scaling properties of Sobolev spaces. For the lifting problem of maps in \(W^{s, p} (\mathcal{M}, \mathbb{S}^1)\), it is known that when \(s \in (0, 1)\) and \(p \in (1, + \infty)\) with \(sp>1\), there exists a sequence of maps \((u_n)_{n \in \N}\) in \(W^{s, p} (\mathcal{M}, \mathbb{S}^1)\) such that \citelist{\cite{Merlet2006}*{Theorem 1.1}\cite{MironescuMolnar2015}*{Proposition 5.7}}
\begin{equation}
\label{nonestimateMMM}
\liminf_{n \to \infty} \frac{\mathcal{E}^\mathrm{lift}_{s,p} (u_n)}{\mathcal{E}_{s, p} (u_n)^{1/s}} > 0\quad\text{and}\quad \lim_{n\to\infty}\mathcal{E}_{s,p}(u_n)=+\infty.
\end{equation}
The exponent \(1/s\) in the denominator rules out the possibility of a linear upper bound when \(s < 1\).

\subsection{Superposition operator}
The \emph{superposition} problem asks whether for a given function \(f : \mathcal{N} \to \mathcal{F}\), one has \(f \compose u \in W^{r, q}(\mathcal{M}, \mathcal{F})\) for each \(u \in W^{s, p} (\mathcal{M}, \mathcal{N})\). In analogy to the previous theorems, we have a uniform bound principle:

\begin{theorem}\label{superposition}
Let \(s,r\in (0,1]\), \(p,q\in [1,+\infty)\), \(m\in\N_*\), let \(\mathcal{M}\) be an \(m\)-dimensional Riemannian manifold which is either \(\R^m\) or compact, \(\mathcal{N}\) and \(\mathcal{F}\) be Riemannian manifolds and assume that \(\mathcal{N}\) is connected and, if either \(sp > m\) or \(s = p = m = 1\), compact. If \(rq=sp<m\) and if a measurable map \(f:\mathcal{N}\to\mathcal{F}\) is such that \(f\compose u\in W^{r,q}(\mathcal{M},\mathcal{F})\) whatever \(u\in W^{s,p}(\mathcal{M},\mathcal{N})\), then there exists a constant \(C>0\) such that for every measurable function \(u:\mathbb{B}^m\to\mathcal{N}\), if either \(sp < m\) or \(\mathcal{E}_{s, p} (u) \le 1/C\), then
\[
\mathcal{E}_{r,q}(f\compose u,\mathbb{B}^m)\le C\,\mathcal{E}_{s,p}(u,\mathbb{B}^m).
\]
\end{theorem}

\Cref{superposition} implies that, with the same assumptions and for each \(x, y \in \mathcal{N}\),
\(d_\mathcal{F} (f (x), f (y)) \le C' d_\mathcal{N} (x, y)^{p/q}\) when \(sp < m\) or \(d_{\mathcal{N}}(x, y)\) remains small (see \cref{actingCondition}). In particular, when \(p>q\), the map \(f\) is constant. When \(p = q\), \(f\) is Lipschitz; this necessary condition is well-known for Sobolev functions \citelist{\cite{MarcusMizel1979}\cite{Igari1965}\cite{Bourdaud1993}\cite{BourdaudSickel2011}\cite{Allaoui2009}}.

\subsection{General uniform boundedness principle} 
The similarity of the statements of \cref{extensionthm,approximationthm,liftingthm,superposition}, is not a coincidence, but comes from the common properties of the extension, relaxed, lifting and composition energies, which are nonnegative functionals that do not increase under the restriction of functions.

\begin{definition}[Energy]
The map \(\mathcal{G}\) is an energy over \(\R^m\) with state space \(\mathcal{N}\) whenever \(\mathcal{G}\) maps every open set \(A \subset \R^m\) and every measurable map \(u : A \to \mathcal{N}\) to some \(\mathcal{G}\,(u, A)\in [0, + \infty]\) such that if $A \subseteq B$ are open sets and if the map $u: B \to \mathcal{N}$ is measurable, then one has $\mathcal{G}\,(u\vert_A,A)\leq \mathcal{G}\,(u,B)$.
\end{definition}

For the sake of simplicity, when the map $u:B\to\mathcal{N}$ is measurable and $A\subset B\subset\R^m$ are open, we write $\mathcal{G}\,(u,A)$ rather than $\mathcal{G}\,(u\vert_A,A)$.

\begin{theorem}[Nonlinear uniform boundedness principle]\label{uniform}
Let \(m \in \N_*\), \(s \in (0, 1]\), \(p \in [1, +\infty)\), \(\mathcal{N}\) be a connected Riemannian manifold which, if either \(sp > m\) or \(s = p = m = 1\), is compact, and let \(\mathcal{G}\) be an  energy over $\R^m$ with state space $\mathcal{N}$. 
Assume that for every measurable map \(u:\R^m\to\mathcal{N}\)
\begin{enumerate}[(i)]
\item (superadditivity) if the sets \(A,B\subset\R^m\) are open and if  \(\Bar{A}\cap \Bar{B} = \emptyset\), then
\[
\mathcal{G}\,(u,A\cup B)\geq \mathcal{G}\,(u,A)+\mathcal{G}\,(u,B),
\]

\item (scaling) for all $\lambda>0$, $h\in\R^m$ and any open set $A\subset\R^m$,
\[
\mathcal{G}\,(u ,h+\lambda A)=\lambda^{m-sp} \mathcal{G}\,(u(h + \lambda \cdot),A).
\]
\end{enumerate}
If for every measurable function \(u : \mathbb{B}^m \to \mathcal{N}\), \(\mathcal{E}_{s, p} (u, \mathbb{B}^m) < + \infty\) implies \(\mathcal{G}\,(u,\mathbb{B}^m) < + \infty\) and \(\mathcal{E}_{s, p} (u, \mathbb{B}^m) = 0\) implies \(\mathcal{G}\,(u,\mathbb{B}^m) = 0\), then there exists a constant $C \in [0, +\infty)$ such that for every measurable map $u:\mathbb{B}^m \to \mathcal{N}$, if either \(sp < m\) or \(\mathcal{E}_{s, p} (u, \mathbb{B}^m) \le 1/C\),
\begin{equation*}
\label{conclusion}
\mathcal{G}\,(u,\mathbb{B}^m)\leq C \,\mathcal{E}_{s,p}(u,\mathbb{B}^m).
\end{equation*}
\end{theorem}

Compared to the statements of the classical uniform boundedness principle in Banach spaces, the nonlinear uniform boundedness principle of \cref{uniform} replaces the linearity assumption  with some superadditivity and some scaling assumption. When dealing with functions spaces, the scaling in the linear target has been replaced by a scaling in the domain. 

Equivalently, \cref{uniform} is a general tool to construct a counterexample out of the failure of a linear estimate.
When \(sp \le m\), these counterexamples form in fact a \emph{dense set} (\cref{theoremDenseCounterexamples}). Similar density of counterexamples have been obtained recently for the Lavrentiev phenomenon for harmonic maps \cite{MazowieckaStrzelecki2017}.  When \(sp \le m\) and the energy \(\mathcal{G}\) is lower semi-continuous, \cref{uniform} and its consequences \cref{extensionthm,superposition,approximationthm,liftingthm} still hold under the weaker assumption that the set \(\{u \in W^{s, p} (\mathbb{B}^m, \mathcal{N}) \st \mathcal{G} (u, \mathbb{B}^m) < +\infty\}\) has at least one interior point in \(W^{s, p}(\mathbb{B}^m, \mathcal{N})\) (see \cref{theoremDenseCounterexamples} below).

If the energy \(\mathcal{G}\) is \emph{lower semi-continuous} --- which is indeed the case in all the examples considered in the present work --- then either a linear energy bound holds or the set of maps in \(W^{s,p}(\mathbb{B}^m,\mathcal{N})\) of infinite energy is a dense countable intersection of open sets, and thus this set is \emph{comeagre} in the sense of \emph{Baire} whereas the set of maps whose energy \(\mathcal{G}\) is finite is \emph{meagre}.

Following the strategy of Hang and Lin \cite{Hang:2003} (see also \citelist{\cite{Bethuel2014Weak}\cite{Bethuel2014Extension}}), \cref{uniform} will be proved by assuming by contradiction the existence for each \(n \in \N\) of a Sobolev map $u_n\in W^{s,p}(\mathbb{B}^m,\mathcal{N})$ such that $\mathcal{G}\,(u_n,\mathbb{B}^m)\geq 2^n \mathcal{E}_{s,p}(u_n,\mathbb{B}^m)$ and then reaching a contradiction by constructing a map \(u \in W^{s, p} (\mathbb{B}^m, \mathcal{N})\) such that \(\mathcal{G}\,(u, \mathbb{B}^m) = + \infty\) in two main constructions:
\begin{description}
\item[Opening] 
The sequence \((u_n)_{n \in \N}\) is transformed by an opening of maps (in the sense of Brezis and Li \cite{BrezisLi2001}) and some gluing of maps in a sequence \((\Tilde{u}_n)_{n \in \N}\) of maps that all take a fixed value near the boundary (see Steps 1--3 in the proof of \cref{theoremCounterexample}, Section \ref{proofUBP}).
\item[Patching] We patch together rescaled translations of the elements of the sequence \((\Tilde{u}_n)_{n \in \N}\) in such a way that they fit together in the unit ball, the total Sobolev energy remains bounded (by a kind of sub-additivity property: see Lemma \ref{subadditive}) but, by superadditivity, the energy \(\mathcal{G}\) is infinite (see Step 4 in the proof of Theorem \ref{theoremCounterexample}, Section \ref{proofUBP}).
\end{description}

A substantial contribution in the present work is the possibility to handle the fractional case \(0 < s < 1\).

The global strategy of the proof of \cref{uniform} is also somehow reminiscent of the original proofs of Hahn and Banach of the uniform boundedness principle, where worse and worse elements are summed up by the gliding hump technique to obtain a contradiction \citelist{\cite{Banach1922}\cite{Hahn1922}} (see also \cite{Sokal2011}).

When \(0 < s < 1\), the proof only uses the fact that \(\mathcal{N}\) is a \emph{Lipschitz-connected metric space}, that is a metric space of which any pair of points is connected by a Lipschitz-continuous path.

The strategy of proof also covers the case \(s = 0\), corresponding to superposition operators in \(L^p\) spaces (see \cref{case_s_0}) and the case \(s > 1\), for which the resulting theorem involves an estimate by the Sobolev on a larger ball and a lower-order term (see \cref{higher}). 

\subsection{Structure of the article} 
\Cref{toolbox} is devoted to the two main tools we need: opening lemma and weak subadditivity of Sobolev energies. We use them in \cref{proofUBP} to prove our general uniform bound principle and we give several applications in \cref{concreteUBP} including \cref{extensionthm,approximationthm,liftingthm,superposition}. We then investigate the generalization of our method to the limiting case \(s = 0\) (\cref{case_s_0}) and to higher order Sobolev spaces (\cref{higher}).

\section{Toolbox}
\label{toolbox}

\subsection{Opening of Sobolev maps}
The aim of the opening construction, introduced by Brezis and Li \cite{BrezisLi2001}, is to perform a singular composition of a Sobolev map \(u \in W^{s, p} (\mathcal{M}, \mathcal{N})\) with a smooth function: given a smooth function \(\varphi\), we want to control the composite map \(u \compose \varphi\) in Sobolev energy. For a fixed change of variable \(\varphi\) which is not a diffeomorphism, in general \(u \compose \varphi\) has infinite energy. It turns out however that it has finite energy if we take \(\varphi\) out of a suitable family of changes of variable.

Since the image under \(\varphi\) of sets of positive Lebesgue measure can be negligible, the singular composition does not preserve equivalence classes of maps equal almost everywhere. 
In order to avoid this problem, we will not put our maps in equivalence classes and we will consider measurable maps \emph{defined everywhere} in their domain.

\begin{lem}[Opening of maps]
\label{openingLemma}
Let \(m \in \N_*\), \(s \in (0, 1]\), \(p \in [1, + \infty)\), \(\lambda > 1\) and \(\eta \in (0, \lambda)\).
There is a constant \(C > 0\) such that for every \(\rho > 0\), every measurable map \(u : \mathbb{B}^m_{\lambda \rho} \to \mathcal{N}\) and every Lipschitz-continuous map \(\varphi : \mathbb{B}^m_{(1 + \eta) \rho} \to \mathbb{B}^m_{(\lambda - \eta) \rho}\), there exists a point \(a \in \mathbb{B}^m_{\eta \rho}\) such that 
\[
 \mathcal{E}_{s, p} \bigl(u \compose (\varphi (\cdot - a) +a ), \mathbb{B}^m_\rho\bigr)
 \le C\, \mathrm{Lip}(\varphi)^{sp}\mathcal{E}_{s, p} (u, \mathbb{B}^m_{\lambda \rho}),
\]
where for every \(r\ge 0\), \(\mathbb{B}^m_r:=\{x\in\R^m\st |x|\le r\}\).
\end{lem}

In the statement the dependence of the point \(a\) on the map \(u\) is essential; modifying \(u\) merely on a Lebesgue null set could change the choice of this point \(a\).

The assumptions on the map \(\varphi\) ensure that if \(a \in \mathbb{B}^m_{\eta \rho}\) and \(x \in \mathbb{B}^m_{\rho}\), then \(\varphi (x - a) + a \in \mathbb{B}^m_{\lambda \rho}\) and thus the left-hand side of the inequality is well defined.

\begin{proof}[Proof of \cref{openingLemma}]
\resetconstant
We define for each point \(a \in \mathbb{B}^m_{\eta \rho}\) the map \(\varphi_a = (\varphi (\cdot - a) + a) : \mathbb{B}^m_{\rho} \to \mathbb{B}^m_{\lambda \rho}\).
We will prove an averaged estimate
\begin{equation}
  \label{meanenergy}
  \fint_{\mathbb{B}^m_{\eta \rho /2 }} \mathcal{E}_{s,p}(u \compose \varphi_a, \mathbb{B}^m_\rho) \diff a \leq C\, \mathrm{Lip}(\varphi)^{sp} \mathcal{E}_{s,p}(u,\mathbb{B}^m_{\lambda \rho}).
\end{equation}
\emph{In the case \(s = 1\)}, we follow \cite{BousquetPonceVanSchaftingen2015}*{Lemma 2.3}: by the chain rule for Sobolev functions, we have $\abs{D (u \compose \varphi_a)} \leq \mathrm{Lip}(\varphi)\, \abs{Du}\compose \varphi_a$ in \(\mathbb{B}^m_\rho\)  and so by definition of the map \(\varphi_a\),
\[
\begin{split}
  \int_{\mathbb{B}^m_{\eta \rho}} \mathcal{E}_{1,p}(u \compose \varphi_a,\mathbb{B}^m_\rho)\diff a
  &=\int_{\mathbb{B}^m_{\eta \rho}} \int_{\mathbb{B}^m_\rho}|D (u \compose \varphi_a) (x)|^p\diff x\diff a\\
  &\leq\mathrm{Lip}(\varphi)^p \int_{\mathbb{B}^m_{\eta \rho}} \biggl(\int_{\mathbb{B}^m_\rho}|Du (a+\varphi(x-a))|^p\diff x\biggr)\diff a
\end{split}
\]
By a change of variable \(y = x - a\) and by interchanging the order of integration, we deduce that 
\[
\begin{split}
\int_{\mathbb{B}^m_{\eta \rho}} \mathcal{E}_{1,p}(u \compose \varphi_a,\mathbb{B}^m_\rho)\diff a
  &\leq\mathrm{Lip}(\varphi)^p \int_{\mathbb{B}^m_{\eta \rho}} \biggl(\int_{\mathbb{B}^m_{(1 + \eta)\rho}}|Du (a+\varphi(y))|^p\diff y\biggr)\diff a\\
  &= \mathrm{Lip}(\varphi)^p \int_{\mathbb{B}^m_{(1 + \eta) \rho}} \mathcal{E}_{1, p} \bigl(u, \mathbb{B}^m_{\eta \rho} (\varphi (y))\bigr) \diff y.
\end{split}
\]
We finally have, by monotonicity of the Sobolev energy,
\[
\begin{split}
  \int_{\mathbb{B}^m_{\eta \rho}} \mathcal{E}_{1,p}(u \compose \varphi_a,\mathbb{B}^m_\rho)\diff a&\le \mathrm{Lip}(\varphi)^p \int_{\mathbb{B}^m_{(1 + \eta) \rho}} \mathcal{E}_{1, p} \bigl(u, \mathbb{B}^m_{\lambda  \rho}\bigr) \diff y\\
  & = \mathcal{L}^m \bigl(\mathbb{B}^m_{(1 + \eta) \rho}\bigr) \,\mathrm{Lip} (\varphi)^p\; \mathcal{E}_{1,p} \bigl(u,\mathbb{B}^m_{\lambda \rho}\bigr).
\end{split}
\]
The conclusion follows with \(C = 2^m (1 + \frac{1}{\eta})^m\).

\emph{When $0<s<1$}, we define for \(x, y \in \mathbb{B}^m_{\rho}\) and \(a \in \mathbb{B}^m_{\eta \rho/2}\) the set
\[
  \mathbb{B}^m_{a,x,y}:= B_{\frac{|\varphi_a(x)-\varphi_a(y)|}{\beta}}\Bigl(\frac{\varphi_a(x)+\varphi_a(y)}{2}\Bigr)\subset \R^m,\quad
  \text{ with }\beta:=\frac{4 \lambda}{\eta} - 2.
\]
For such points \(x,y,a\), we observe that \(|\varphi_a(x)\pm\varphi_a(y)|\leq \rho(2\lambda-\eta)\). In particular,
\[
 \frac{\abs{\varphi_a(x)+\varphi_a(y)}}{2} + \frac{\abs{\varphi_a(x)-\varphi_a(y)}}{\beta}
 \le 
 \rho (2\lambda-\eta) \Bigl(\frac 12+\frac 1\beta\Bigr) 
 = \lambda \rho,
\]
and therefore \(\mathbb{B}^m_{a, x, y} \subseteq \mathbb{B}^m_{\lambda \rho}\).
By the triangle inequality and by convexity of the function \(t \in \R \mapsto \abs{t}^p\), we have for every \(x, y \in \mathbb{B}^m_{\rho}\), \(a \in \mathbb{B}^m_{\eta\rho/2}\) and \(z \in \mathbb{B}^m_{a,x,y}\) ,
\[
d_\mathcal{N}\big(u \compose \varphi_a(x), u \compose \varphi_a(y)\big)^p\leq 2^{p - 1} \bigl(d_\mathcal{N}\big(u \compose \varphi_a(x),u(z)\big)^p+ d_\mathcal{N}\big(u(z),u \compose \varphi_a(y)\big)^p\bigr). 
\]
By averaging over \(z \in \mathbb{B}^m_{a, x, y}\), we obtain
\[
\begin{split}
 \mathcal{E}_{s,p}(u \compose \varphi_a,\mathbb{B}^m_\rho) 
 & = \int_{\mathbb{B}^m_{\rho}} \int_{\mathbb{B}^m_\rho} \frac{d_\mathcal{N}\big(u (\varphi_a (x)), u (\varphi_a (y))\big)^p}{\abs{x - y}^{m + sp}} \diff x \diff y\\
 &\leq 2^{p}  \int_{\mathbb{B}^m_{\rho}} \int_{\mathbb{B}^m_\rho} \fint_{\mathbb{B}^m_{a,x,y}}\frac{d_{\mathcal{N}}\big(u(\varphi_a(x)),u(z)\big)^p}{\abs{x - y}^{m+sp}}\diff z\diff x\diff y\\
 & = \C \int_{\mathbb{B}^m_{\rho}} \int_{\mathbb{B}^m_\rho} \int_{\mathbb{B}^m_{a,x,y}}\frac{d_{\mathcal{N}}\big(u(\varphi_a(x)),u(z)\big)^p}{\abs{\varphi_a (x) - \varphi_a (y)}^m \abs{x - y}^{m+sp}}\diff z\diff x\diff y.
\end{split}
\]
We next observe that if \(z \in \mathbb{B}^m_{a, x, y}\)
\[
\begin{split}
 \abs{\varphi_a (x) - z} &\le 
  \Bigabs{\frac{\varphi_a (x) - \varphi_a (y)}{2} } 
  +\Bigabs{\frac{\varphi_a (x) + \varphi_a (y)}{2} - z} \\
 &\le \Big(\frac 12+\frac 1\beta\Big) \abs{\varphi_a (x) - \varphi_a (y)}
  = \frac{\lambda}{2\lambda-\eta} \abs{\varphi_a (x) - \varphi_a (y)},
\end{split}
\]
and therefore
\[
\begin{split}
 \mathcal{E}_{s,p}(u \compose \varphi_a,\mathbb{B}^m_\rho)
 & \le\Cl{ixye} \int_{\mathbb{B}^m_{\rho}} \int_{\mathbb{B}^m_\rho} \int_{\mathbb{B}^m_{a,x,y}}\frac{d_{\mathcal{N}}\big(u(\varphi_a(x)),u(z)\big)^p}{\abs{\varphi_a (x) - z}^m \abs{x - y}^{m+sp}}\diff z\diff x\diff y.
\end{split}
\]
By Fubini's theorem, this can be rewritten as
\begin{equation}
\label{averagedYestimate}
\begin{split}
 \mathcal{E}_{s,p}(u \compose \varphi_a,\mathbb{B}^m_\rho)
 & \le \Cr{ixye} \int_{\mathbb{B}^m_{\lambda\rho}} \int_{\mathbb{B}^m_\rho} \int_{Y_{a,x, z}}\frac{d_{\mathcal{N}}\big(u(\varphi_a(x)),u(z)\big)^p}{\abs{\varphi_a (x) - z}^m \abs{x - y}^{m+sp}}\diff y\diff x\diff z,
\end{split}
\end{equation}
where $Y_{a,x,y} \subset \R^m$ is the set of points $y$ for which $z\in \mathbb{B}^m_{a,x,y}$:
\[
 Y_{a, x, z} = \bigl\{ y \in \mathbb{B}^m_{\rho} \st \beta \abs{\varphi_a (x) + \varphi_a (y) - 2 z} \le 2\abs{\varphi_a (x) - \varphi_a (y)}\bigr\}.
\]
Since \(\abs{\varphi_a (x) + \varphi_a (y) - 2 z} \ge 2 \abs{\varphi_a (x) - z} -\abs{\varphi_a (x) - \varphi_a (y)}\), we have 
\[
\begin{split}
 Y_{a, x, z} &\subseteq\bigl\{ y \in \mathbb{B}^m_{\rho} \st \abs{\varphi_a (x) - z} \le \Cl{ixlpn} \abs{\varphi_a (x) - \varphi_a (y)}\bigr\}\\
 & \subseteq  \bigl\{ y \in \R^m \st \abs{\varphi_a (x) - z} \le \Cr{ixlpn} \, \mathrm{Lip} (\varphi) \abs{x - y}\bigr\}.
\end{split}
\]
We compute
\begin{equation}
\label{estimateYintegral}
\int_{Y_{a, x, z}}\frac{\diff y}{\abs{x - y}^{m+sp}}\le 
\int_{\R^m \setminus B_{\frac{\abs{\varphi_a (x) - z}}{C_3 \mathrm{Lip} (\varphi)}}(x)}\frac{\diff y}{\abs{x - y}^{m+sp}}= 
\C \frac{\mathrm{Lip} (\varphi)^{sp}}{\abs{\varphi_a (x) - z}^{sp}}.
\end{equation}
By combining \eqref{averagedYestimate} and \eqref{estimateYintegral}, integrating over \(a\in \mathbb{B}^m_{\eta\rho/2}\) and by the changes of variable \(y = x - a\in \mathbb{B}^m_{(1+\frac \eta 2)\rho}\) and \(w=a+\varphi(y)\in \mathbb{B}^m_{(\lambda-\frac\eta 2)\rho}\), we are led to the estimates
\begin{equation*}
\begin{split}
\int_{\mathbb{B}^m_{\eta \rho/2}} \mathcal{E}_{s,p}(u \compose \varphi_a,\mathbb{B}^m_\rho) \diff a
&\leq \Cl{irnxx} \mathrm{Lip}(\varphi)^{sp} \int_{\mathbb{B}^m_{\eta \rho/2}}\int_{\mathbb{B}^m_{\lambda \rho}}\int_{\mathbb{B}^m_{\rho}}\frac{d_{\mathcal{N}}\big(u(\varphi_a(x)),u(z)\big)^p}{|\varphi_a(x)-z|^{m+sp}}\diff x\diff z \diff a\\
&\leq \Cr{irnxx} \mathrm{Lip}(\varphi)^{sp} \int_{\mathbb{B}^m_{\eta \rho/2}} \int_{\mathbb{B}^m_{\lambda  \rho}}\int_{\mathbb{B}^m_{(1 + \frac{\eta}{2})\rho}}\frac{d_\mathcal{N}\big(u(a+\varphi(y)), u(z)\big)^p}{|a+\varphi(y)-z|^{m+sp}}\diff y\diff z \diff a\\
&\leq \Cr{irnxx} \mathrm{Lip}(\varphi)^{sp} \int_{\mathbb{B}^m_{(1 + \frac{\eta}{2}) \rho}} \int_{\mathbb{B}^m_{(\lambda - \frac{\eta}{2})\rho}} \int_{\mathbb{B}^m_{\lambda \rho}} \frac{d_\mathcal{N}\big(u(w), u(z)\big)^p}{|w-z|^{m+sp}}\diff z \diff w \diff y\\
& \leq \Cr{irnxx}\mathrm{Lip}(\varphi)^{sp} \mathcal{L}^m (\mathbb{B}^m_{(1 + \frac{\eta}{2}) \rho})\, \mathcal{E}_{s, p} (u, \mathbb{B}^m_{\lambda \rho}).
\end{split}
\end{equation*}
The conclusion follows with \(C=\Cr{irnxx}(1+\frac 2\eta)^m\).
\end{proof}

\subsection{Gluing interior and exterior estimates}
The next lemma will allow us to combine constructions performed on different parts of the domain. Whereas when \(s = 1\) it is sufficient to have traces matching on the boundary, the nonlocality of the fractional case \(s \in (0, 1)\) invites us to consider a gluing with a buffer zone \(\mathbb{B}^m_{\rho} \setminus \Bar{\mathbb{B}}^m_{\eta \rho}\) in the energies. 

\begin{lem}[Gluing along a buffer zone]
\label{lemmaBallToRN}
Let \(m \in \N_*\), \(s \in (0, 1]\), \(p \in [1, \infty)\). There exists a constant \(C > 0\) such that for every \(\eta \in (0, 1)\), every open set \(A \subset \R^m\), every measurable function \(u : A \to \mathcal{N}\) and every \(\rho > 0\) such that \(\mathbb{B}^m_{\rho} \setminus \Bar{\mathbb{B}}^m_{\eta \rho} \subseteq A\),
\[
  \mathcal{E}_{s,p}(u, A)\leq \Bigl(1 + \frac{C}{(1 - \eta)^{sp + 1}} \Bigr)\mathcal{E}_{s,p}(u,A \cap \mathbb{B}^m_{\rho})
  + \Bigl(1 + \frac{C\eta^m}{1 - \eta}\Bigr) \mathcal{E}_{s,p}(u,A \setminus \Bar{\mathbb{B}}^m_{\eta \rho}). 
\]
\end{lem}

The constant \(C\) in the statement of \cref{lemmaBallToRN} only depends on the dimension \(m\), on the regularity \(s \in (0, 1]\) and on the integrability \(p \in [1, +\infty)\). It does not depend on the set \(A\) nor on the map \(u\) nor on the radius \(\rho\) nor on \(\eta\). 

We will apply \cref{lemmaBallToRN} in the case where \(A\) is the entire Euclidean space \(\R^m\) and a ball \(\mathbb{B}^m_R \subset \R^m\) with \(\rho < R\).

\begin{proof}[Proof of \cref{lemmaBallToRN}]
\resetconstant 
When \(s = 1\), we have \(\mathcal{E}_{s,p}(u,A) \le \mathcal{E}_{s,p}(u,A \cap \mathbb{B}^m_{\rho}) + \mathcal{E}_{s, p} (u,A \setminus \Bar{\mathbb{B}}^m_{\eta \rho})\) and the conclusion follows with \(C = 1\).

For \(0 < s < 1\), we have by additivity of the double integral defining the fractional Sobolev energy \(\mathcal{E}_{s, p}\),
\begin{equation}
\label{ineqEnergyBalanceInnerOuter}
 \mathcal{E}_{s,p}(u,A) 
\le  \mathcal{E}_{s, p} (u, A \cap \mathbb{B}^m_{\rho})+\mathcal{E}_{s, p} (u, A \setminus \Bar{\mathbb{B}}^m_{\eta \rho})
 + 2\int_{A \setminus \mathbb{B}^m_{\rho}} 
 \int_{A \cap \Bar{\mathbb{B}}^m_{\eta \rho}} \frac{d_\mathcal{N}\bigl(u (x), u (y)\bigr)^p}{\abs{ x - y}^{m + sp}} \diff x \diff y;
\end{equation}
it will thus be sufficient to estimate the last integral on the right-hand side.
We notice that for each \(x \in A \cap \Bar{\mathbb{B}}^m_{\eta \rho}\), \(y \in A \setminus \mathbb{B}^m_{\rho}\) and \(z \in \mathbb{B}^m_{\rho} \setminus \Bar{\mathbb{B}}^m_{\eta \rho}\subset A\), we have, by convexity of the function \(t \in \R \mapsto \abs{t}^p\),
\begin{equation}
\label{ineqyxz}
  d_\mathcal{N}(u(x), u(y))^p \le 2^{p-1} \bigl( d_{\mathcal{N}} \bigl(u(x),u(z)\bigr)^p+d_\mathcal{N} \bigl(u(z),u(y)\bigr)^p\bigr).
\end{equation}
By averaging the inequality  \eqref{ineqyxz} over \(z \in \mathbb{B}^m_{\rho} \setminus \Bar{\mathbb{B}}^m_{\eta \rho}\) we estimate the integral in the right-hand side of \eqref{ineqEnergyBalanceInnerOuter} as 
\begin{multline}
\label{ineqyxzaveraged}
 \int_{A \setminus \mathbb{B}^m_{\rho}} 
 \int_{A \cap \Bar{\mathbb{B}}^m_{\eta \rho}} \frac{d_\mathcal{N}\bigl(u (x), u (y)\bigr)^p}{\abs{ x - y}^{m + sp}} \diff x \diff y\\
 \le  2^{p-1}\biggl(\fint_{\mathbb{B}^m_\rho \setminus \Bar{\mathbb{B}}^m_{\eta \rho}} \int_{A \setminus \mathbb{B}^m_{\rho}} 
 \int_{A \cap \Bar{\mathbb{B}}^m_{\eta \rho}} \frac{d_\mathcal{N} \bigl(u (x), u (z)\bigr)^p}{\abs{ x - y}^{m + sp}} \diff x \diff y \diff z \\
 + \fint_{\mathbb{B}^m_\rho \setminus \Bar{\mathbb{B}}^m_{\eta \rho}} \int_{A \setminus \mathbb{B}^m_{\rho}} 
 \int_{A \cap \Bar{\mathbb{B}}^m_{\eta \rho}} \frac{d_\mathcal{N} \bigl(u (z), u (y)\bigr)^p}{\abs{ x - y}^{m + sp}} \diff x \diff y \diff z \biggr).
\end{multline}
For the first integral in the right-hand side of \eqref{ineqyxzaveraged}, since for each \(x \in A \cap \Bar{\mathbb{B}}^m_{\eta \rho}\) and \(y\in A\setminus\mathbb{B}^m_{\rho}\), one has \(|x-y|\ge (1-\eta)\rho\), we first have by integration over \(y\)
\begin{multline*}
 \int_{\mathbb{B}^m_\rho \setminus \Bar{\mathbb{B}}^m_{\eta \rho}} \int_{A \setminus \mathbb{B}^m_{\rho}} 
 \int_{A \cap \Bar{\mathbb{B}}^m_{\eta \rho}} \frac{d_\mathcal{N} \bigl(u (x), u (z)\bigr)^p}{\abs{ x - y}^{m + sp}} \diff x \diff y \diff z\\
 \le \frac{\C}{(1 - \eta)^{s p} \rho^{s p}}  \int_{\mathbb{B}^m_\rho \setminus \Bar{\mathbb{B}}^m_{\eta \rho}} 
 \int_{A \cap  \Bar{\mathbb{B}}^m_{\eta \rho}} d_\mathcal{N} \bigl(u (x), u (z)\bigr)^p \diff x \diff z.
\end{multline*}
Moreover, by dividing by the measure of the set \(\mathbb{B}^m_\rho \setminus \Bar{\mathbb{B}}^m_{\eta \rho}\) and noting that for each \(x \in A \cap\Bar{\mathbb{B}}^m_{\eta \rho}\) and \(z \in \mathbb{B}^m_{\rho} \setminus \Bar{\mathbb{B}}^m_{\eta \rho}\), one has \(\abs{x - z} \le 2 \rho\), we conclude that  
\begin{equation}
\label{ineqInnerOuterFirstIntegral}
\begin{split}
 \fint_{\mathbb{B}^m_\rho \setminus \Bar{\mathbb{B}}^m_{\eta \rho}} \int_{A \setminus \mathbb{B}^m_{\rho}} 
 \int_{A \cap \Bar{\mathbb{B}}^m_{\eta \rho}} &\frac{d_\mathcal{N} \bigl(u (x), u (z)\bigr)^p}{\abs{ x - y}^{m + sp}} \diff x \diff y \diff z\\
 &\le \frac{\Cl{instr}}{(1 - \eta)^{sp}(1-\eta^m)} \int_{\mathbb{B}^m_\rho \setminus \Bar{\mathbb{B}}^m_{\eta \rho}} 
 \int_{A \cap \Bar{\mathbb{B}}^m_{\eta \rho}} \frac{d_\mathcal{N} \bigl(u (x), u (z)\bigr)^p}{\abs{ x - z}^{m + sp}} \diff x \diff z\\
 &\le \frac{\Cr{instr}}{(1 - \eta)^{sp+1} } \,\mathcal{E}_{s, p} (u, A \cap \mathbb{B}^m_\rho).
\end{split}
\end{equation}

We consider now the second integral in the right-hand side of \eqref{ineqyxzaveraged}. We note that if \(x \in A \cap \Bar{\mathbb{B}}^m_{\eta \rho}\) and \(y \in A \setminus \mathbb{B}^m_\rho\), then \(\abs{x - y} \ge \abs{y} - \eta \rho \ge 0\), and thus 
\begin{equation*}
\int_{\mathbb{B}^m_\rho \setminus \Bar{\mathbb{B}}^m_{\eta \rho}} \int_{A \setminus \mathbb{B}^m_{\rho}} 
 \int_{A \cap \Bar{\mathbb{B}}^m_{\eta \rho}} \frac{d_\mathcal{N} \bigl(u (z), u (y)\bigr)^p}{\abs{ x - y}^{m + sp}} \diff x \diff y \diff z
 \le \C \eta^m \rho^m \int_{\mathbb{B}^m_\rho \setminus \Bar{\mathbb{B}}^m_{\eta \rho}} \int_{A \setminus \mathbb{B}^m_{\rho}}\frac{
 d_\mathcal{N} \bigl(u (z), u (y)\bigr)^p}{(\abs{y} - \eta \rho)^{m + sp}}\diff y \diff z.
\end{equation*}
Next, if \(y \in A \setminus \mathbb{B}^m_{\rho}\) and \(z \in \mathbb{B}^m_\rho \setminus \Bar{\mathbb{B}}^m_{\eta \rho}\), we have
\[
\abs{y - z} \le \mathrm{dist}(y,\mathbb{B}^m_{\eta\rho})+\mathrm{dist}(z,\mathbb{B}^m_{\eta\rho})\le 2\,\mathrm{dist}(y,\mathbb{B}^m_{\eta\rho})=2(|y|-\eta\rho)
\]
and therefore 
\begin{equation}
\label{ineqInnerOuterSecondIntegral}
 \begin{split}
 \fint_{\mathbb{B}^m_\rho \setminus \Bar{\mathbb{B}}^m_{\eta \rho}} \int_{A \setminus \mathbb{B}^m_{\rho}} 
 \int_{A \cap \Bar{\mathbb{B}}^m_{\eta \rho}} \frac{d_\mathcal{N} \bigl(u (z), u (y)\bigr)^p}{\abs{ x - y}^{m + sp}} \diff x \diff y \diff z
 &\le \frac{\Cl{instr2} \eta^m}{1 - \eta^m} \int_{\mathbb{B}^m_\rho \setminus \Bar{\mathbb{B}}^m_{\eta \rho}} \int_{A \setminus \mathbb{B}^m_{\rho}} 
 \frac{d_\mathcal{N} \bigl(u (z), u (y)\bigr)^p}{\abs{z - y}^{m + sp}} \diff y \diff z\\
 &\le \frac{\Cr{instr2} \eta^m}{1 - \eta} \mathcal{E}^{s, p} \bigl(u, A \setminus \Bar{\mathbb{B}}^m_{\eta \rho}\bigr).
\end{split}
\end{equation}
The conclusion follows then from \eqref{ineqEnergyBalanceInnerOuter}, \eqref{ineqyxzaveraged}, \eqref{ineqInnerOuterFirstIntegral} and \eqref{ineqInnerOuterSecondIntegral} with \(C = 2 \max\{\Cr{instr},  \Cr{instr2}\}\).
\end{proof}
 
\subsection{Patching countably many Sobolev maps}
We want to estimate the energy of a map obtained by patching countable many maps different from a common constant value on disjoint sets $A_i$. If we apply the gluing technique from above (\cref{lemmaBallToRN}) countably many times (which essentially means, for each $i$, estimating the total energy of $u$ by the energy on $A_i$ plus the energy out of $A_i$), since the constants appearing in the statement are larger than \(1\) when \(s \in (0, 1)\), the constant coming from the iterative process will be unbounded and will thus give no estimate in the limit. In order to deal with this situation, we derive a specific bound for the patching of a countable family of maps.

\begin{lem}[Countable patching]\label{subadditive}
Let \(m \in \N_*\), \(s \in (0, 1]\), \(p \in [1, + \infty)\), let \({\mathcal{M}}\) be a Riemannian manifold, \(I\) be a finite or countably infinite set, and for each \(i\in I\), let \(u_i : {\mathcal{M}} \to \mathcal{N}\) be a measurable map. If there exist \(b \in \mathcal{N}\) and a collection \((A_i)_{i \in I}\) of open subsets of \({\mathcal{M}}\) such that if \(x \in {\mathcal{M}} \setminus A_i\) with \(i \in I\), \(u_i (x) = b\) and if \(i, j \in I\) with  \(i \neq j\), \(\Bar{A}_i \cap \Bar{A}_j = \emptyset\), then, if \(u : \mathcal{M} \to \mathcal{N}\) is defined by 
\[
 u (x) = \begin{cases}
          u_i (x) & \text{if \(x \in A_i\)},\\
          b & \text{otherwise},
         \end{cases}
\]
we have 
\[
\mathcal{E}_{s,p}(u,{\mathcal{M}})\leq C\sum_{i\in I}\mathcal{E}_{s,p}(u_i,{\mathcal{M}})
\]
with \(C = 1\) if \(s = 1\) and \(C = 2^p\) if \(s \in (0, 1)\).
\end{lem}

\begin{proof}
First, we consider the case where \(s = 1\) and the set \(I\) is finite. For each \(i \in I\), if \(\mathcal{E}_{1, p} (u_i, {\mathcal{M}}) < + \infty\), then the function \(u\) is weakly differentiable on the set \({\mathcal{M}} \setminus \bigcup_{j \in I \setminus \{i\}} \Bar{A}_j\), which is open since \(I\) is finite, and its energy on this set is controlled by \(\mathcal{E}_{1,p}(u_i,{\mathcal{M}})\). Therefore the function \(u\) is weakly differentiable and, by additivity of the integral, we have
\[
 \mathcal{E}_{1, p} (u,{\mathcal{M}}) \le \sum_{i \in I} \mathcal{E}_{1, p} (u_i,{\mathcal{M}}).
\]
If the set \(I\) is countably infinite, we can write that \(I = \mathbb{N}\) and we can define 
\[
  u^n (x) = \begin{cases}
          u_i (x) & \text{if \(x \in A_i\) and \(i \in \{0, \dotsc, n\}\)},\\
          b & \text{otherwise},
         \end{cases}
\]
By the first part of the proof
\[
\mathcal{E}_{1, p} (u^n,{\mathcal{M}}) \le \sum_{i \in \{0, \dotsc, n\}} \mathcal{E}_{1, p} (u_i,{\mathcal{M}}) \le \sum_{i \in \mathbb{N}} \mathcal{E}_{1, p} (u_i,{\mathcal{M}}).
\]
The conclusion follows from the fact that \((u^n)_{n \in \N}\) converges almost everywhere to \(u\) and the lower semi-continuity of the Sobolev energy under the almost everywhere convergence. 

\medskip

We assume now that \(0 < s < 1\).
We write, by additivity of the integral and the fact that \(u (x) = u (y)\) if \((x, y) \in ({\mathcal{M}} \setminus \bigcup_{i \in I} A_i)\times ({\mathcal{M}} \setminus \bigcup_{i \in I} A_i)\),
\begin{multline*}
 \mathcal{E}_{s, p} (u, {\mathcal{M}})
 = \sum_{i \in I} \mathcal{E}_{s, p} (u, A_i)
 + \sum_{\substack{i, j \in I\\ i \ne j}} \int_{A_i} \int_{A_j} \frac{d_\mathcal{N} \bigl(u (x), u(y)\bigr)^p}{d_\mathcal{M}(x,y)^{m + sp}} \diff x \diff y\\
 + 2 \sum_{i \in I} \int_{A_i} \int_{{\mathcal{M}} \setminus \bigcup_{j \in I}A_j} \frac{d_\mathcal{N}\bigl(u (x), u (y)\bigr)^p}{d_\mathcal{M}(x,y)^{m + sp}} \diff x \diff y.
\end{multline*}
We first observe that, by assumption, 
\[
 \mathcal{E}_{s, p} (u, A_i) = \mathcal{E}_{s, p} (u_i, A_i),
\]
and 
\[
 \int_{A_i}\int_{{\mathcal{M}} \setminus \bigcup_{j \in I}A_j} \frac{d_\mathcal{N}\bigl(u (x), u (y)\bigr)^p}{d_\mathcal{M}(x,y)^{m + sp}} \diff x \diff y
 =\int_{A_i} \int_{{\mathcal{M}} \setminus \bigcup_{j \in I}A_j} \frac{d_\mathcal{N} \bigl(u_i (x), u_i (y)\bigr)^p}{d_\mathcal{M}(x,y)^{m + sp}} \diff x \diff y.
\]
Finally, if \(i, j \in I\) and \(i \ne j\), we have 
\[
\begin{split}
 \int_{A_i} \int_{A_j} &\frac{d_\mathcal{N} \bigl(u (x), u(y)\bigr)^p}{d_\mathcal{M}(x,y)^{m + sp}} \diff x \diff y\\
 &\le  2^{p - 1}\biggl(\int_{A_i} \int_{A_j} \frac{d_\mathcal{N} \bigl(u (x), b\bigr)^p}{d_\mathcal{M}(x,y)^{m + sp}} \diff x \diff y + \int_{A_i} \int_{A_j} \frac{d_\mathcal{N} \bigl(b, u(y)\bigr)^p}{d_\mathcal{M}(x,y)^{m + sp}} \diff x \diff y\biggr)\\
 & =  2^{p - 1}\biggl(\int_{A_i} \int_{A_j} \frac{d_\mathcal{N}\bigl(u_j (x), u_j (y)\bigr)^p}{d_\mathcal{M}(x,y)^{m + sp}} \diff x \diff y + \int_{A_i} \int_{A_j} \frac{d_\mathcal{N}\bigl(u_i(x), u_i (y)\bigr)^p}{d_\mathcal{M}(x,y)^{m + sp}} \diff x \diff y\biggr).
\end{split}
\]
Therefore, we have 
\begin{multline*}
 \mathcal{E}_{s, p} (u, {\mathcal{M}})
 \le \sum_{i \in I} \mathcal{E}_{s, p} (u_i, A_i)
 +  2^p \sum_{\substack{i, j \in I\\ i \ne j}} \int_{A_i} \int_{A_j} \frac{d_\mathcal{N}\big(u_i (x), u_i(y)\big)^p}{d_\mathcal{M}(x,y)^{m + sp}} \diff x \diff y\\
 + 2 \sum_{i \in I} \int_{A_i} \int_{{\mathcal{M}} \setminus \bigcup_{j \in I}A_j} \frac{d_\mathcal{N}\big(u_i (x), u_i (y)\big)^p}{d_\mathcal{M}(x,y)^{m + sp}} \diff x \diff y,
\end{multline*}
which implies that \(\mathcal{E}_{s,p}(u,{\mathcal{M}})\le 2^{p} \sum_{i \in I} \mathcal{E}_{s, p} (u_i, {\mathcal{M}})\).
\end{proof}

\subsection{Extension}
In the application of the opening construction (\cref{openingLemma}), because the change of variable \(\varphi_a\) is completely known a priori, we will need to define our map \(u\) on a slightly larger domain and with a control on the energy.

\begin{lem}[Extension]
\label{extensionLemma}
Let \(m \in \N_\ast\), \(s \in (0,1]\), \(p \in [1, + \infty)\) and \(\lambda \geq 1\).
There exists \(C > 0\) such that if \(\rho > 0\) and \(u : \mathbb{B}^m_\rho \to \mathcal{N}\) is measurable, there exists \(v : \mathbb{B}^m_{\lambda \rho} \to \mathcal{N}\) such that \(v = u\) on \(\mathbb{B}^m_\rho\) and 
\[
 \mathcal{E}_{s, p} (v, \mathbb{B}^m_{\lambda \rho}) \le C\, \mathcal{E}_{s,p}(u,\mathbb{B}^m_\rho).
\]
\end{lem}
\begin{proof}
This proof is classical. For the convenience of the reader, we sketch an argument based on Euclidean inversion. By scaling we can assume that \(\rho = 1\) and we define the map \(v : \mathbb{B}^m_\lambda \to \mathcal{N}\) by setting
\begin{equation}
\label{eqDefvInversion}
 v (x)
 = \begin{cases}
     u( x) & \text{if \(\abs{x} < 1\)},\\
     u (x/\abs{x}^2) & \text{if \(\abs{x} > 1\)}.
   \end{cases}
\end{equation}
If \(s = 1\), one can check that 
\[
 \mathcal{E}_{1, p} (v, \mathbb{B}^m_\lambda) = \int_{\mathbb{B}^m_1} \abs{D u}^p + \int_{\mathbb{B}^m_1 \setminus \mathbb{B}^m_{1/\lambda}} \frac{\abs{D u (x)}^p}{\abs{x}^{2 (m - p)}} \diff x 
 \le \bigl(1 + \lambda^{2( m - p)_+}\bigr) \, \mathcal{E}_{1, p} (u, \mathbb{B}^m_1).
\]
When \(0 < s < 1\), we have by a change of variable
\begin{multline*}
  \mathcal{E}_{s, p} (v, \mathbb{B}^m_\lambda) = \int_{\mathbb{B}^m_1} \int_{\mathbb{B}^m_1} \frac{d_\mathcal{N}\bigl(u (x), u (y)\bigr)^p}{\abs{x - y}^{m + sp} }\diff x \diff y
  + 2 \int_{\mathbb{B}^m_1} \int_{\mathbb{B}^m_{1/\lambda}} \frac{d_\mathcal{N}\bigl(u (x), u (y)\bigr)^p}{\abs{x/\abs{x}^2 - y}^{m + sp} \abs{x}^{2 m}}\diff x \diff y\\
  + \int_{\mathbb{B}^m_1 \setminus \mathbb{B}^m_{1/\lambda}}\int_{\mathbb{B}^m_1 \setminus \mathbb{B}^m_{1/\lambda}} \frac{d_\mathcal{N}\bigl(u (x), u (y)\bigr)^p}{\abs{x/\abs{x}^2 - y/\abs{y}^2}^{m + sp} \,\abs{x}^{2 m}\,\abs{y}^{2m}}\diff x \diff y.
\end{multline*}
We observe that if \(x, y \in \mathbb{B}^m_1\), we have 
\(
 \abs{x}^2\, \abs{x/\abs{x}^2 - y}^2 = \abs{x - y}^2 + (1 - \abs{x}^2)\,(1 - \abs{y}^2)\ge \abs{x - y}^2
\)
and \(\abs{x}^2\, \abs{y}^2\, \abs{x / \abs{x}^2 - y/\abs{y}^2}^2 = \abs{x - y}^2 \); therefore,
\begin{equation*}
\begin{split}
  \mathcal{E}_{s, p} (v, \mathbb{B}^m_\lambda) &= \int_{\mathbb{B}^m_1} \int_{\mathbb{B}^m_1} \frac{d_\mathcal{N}\bigl(u (x), u (y)\bigr)^p}{\abs{x - y}^{m + sp} }\diff x \diff y
  + 2 \int_{\mathbb{B}^m_1} \int_{\mathbb{B}^m_{1/\lambda}} \frac{d_\mathcal{N}\bigl(u (x), u (y)\bigr)^p}{\abs{x - y}^{m + sp} \abs{x}^{m - sp}}\diff x \diff y\\
  &\qquad \qquad + \int_{\mathbb{B}^m_1 \setminus \mathbb{B}^m_{1/\lambda}}\int_{\mathbb{B}^m_1 \setminus \mathbb{B}^m_{1/\lambda}} \frac{d_\mathcal{N}\bigl(u (x), u (y)\bigr)^p}{\abs{x - y}^{m + sp} \abs{x}^{m - sp}\abs{y}^{m -sp}}\diff x \diff y\\
  &\le \bigl(1 + \lambda^{2(m - sp)_+}\bigr) \,\mathcal{E}_{s,p} (u, \mathbb{B}^m_1).\qedhere
\end{split}
\end{equation*}
\end{proof}

\section{General uniform boundedness principle}%
\label{proofUBP}

\subsection{Obtaining a single obstruction} 
We will prove a slightly refined version of the contraposite of \cref{uniform}.
\begin{theorem}
\label{theoremCounterexample}
Let \(m \in \N_*\), \(s \in (0, 1]\), \(p \in [1, +\infty)\), \(\mathcal{N}\) be a connected Riemannian manifold, and let \(\mathcal{G}\) be an  energy over $\R^m$ with state space $\mathcal{N}$. 
Assume that for every measurable map \(u:\R^m\to\mathcal{N}\)
\begin{enumerate}[(i)]
\item (superadditivity) for all open sets \(A,B\subset\R^m\) with disjoint closure,
\[
\mathcal{G}\,(u,A\cup B)\geq \mathcal{G}\,(u,A)+\mathcal{G}\,(u,B),
\]
\item (scaling) for every $\lambda>0$, every $h\in\R^m$ and every open set $A\subset\R^m$, one has 
\[
\mathcal{G}\,(u ,h+\lambda A)=\lambda^{m-sp} \,\mathcal{G}\,\bigl(u(h + \lambda \cdot),A\bigr).
\]
\end{enumerate}
\begin{description}
\item[Subcritical case]
If \(sp < m\) and if there exists a sequence \((u_n)_{n \in \N}\) of measurable maps from \(\mathbb{B}^m \) to \(\mathcal{N}\) such that for each \(n \in \N\), \(\mathcal{E}_{s, p} (u_n, \mathbb{B}^m) > 0\), \(\mathcal{G}\,(u_n,\mathbb{B}^m)<+\infty\), and 
\[
  \lim_{n \to \infty} \frac{\mathcal{G}\,(u_n, \mathbb{B}^m)}{\mathcal{E}_{s, p} (u_n, \mathbb{B}^m)} = + \infty,
\]
then for every \(b_* \in \mathcal{N}\) and every \(\varepsilon > 0\)  there exists a measurable map \(u : \mathbb{R}^m \to \mathcal{N}\) such that \(\mathcal{E}_{s, p} (u, \mathbb{R}^m) \le \varepsilon \), \(u = b_*\) in \(\R^m \setminus \mathbb{B}^m_{1/2}\) and \(\mathcal{G}\,(u, \mathbb{B}^m) = + \infty\).
\item[Critical case]
If \(sp=m\), \(s < m\), and if there exists a sequence \((u_n)_{n \in \N}\) of measurable maps from \(\mathbb{B}^m\) to \(\mathcal{N}\) such that for each \(n\in\N\),  \(\mathcal{E}_{s,p}(u_n,\mathbb{B}^m)>0\), \(\mathcal{G}\,(u_n,\mathbb{B}^m)<+\infty\), and
\[
\lim_{n\to\infty}\mathcal{E}_{s,p}(u_n,\mathbb{B}^m)=0\quad\text{and}\quad\lim_{n \to \infty} \frac{\mathcal{G}\,(u_n, \mathbb{B}^m)}{\mathcal{E}_{s,p} (u_n, \mathbb{B}^m)} = + \infty,
\]
then for every \(b_* \in \mathcal{N}\) and every \(\varepsilon > 0\), there exists a measurable map \(u : \mathbb{R}^m \to \mathcal{N}\) such that \(\mathcal{E}_{s, p} (u, \mathbb{R}^m) \le \varepsilon\), \(u = b_*\) in \(\R^m \setminus \mathbb{B}^m_{1/2}\) and \(\mathcal{G}\,(u, \mathbb{B}^m) = + \infty\).
\item[Supercritical case]
If \(sp >m\) or \(s = p = m = 1\), if \(\mathcal{N}\) is compact, and if there exists a sequence \((u_n)_{n \in \N}\) of measurable maps from \(\mathbb{B}^m\) to \(\mathcal{N}\) such that for each \(n\in\N\),  \(\mathcal{E}_{s,p}(u_n,\mathbb{B}^m)>0\), \(\mathcal{G}\,(u_n,\mathbb{B}^m)<+\infty\), and
\[
\lim_{n\to\infty}\mathcal{E}_{s,p}(u_n,\mathbb{B}^m)=0\quad\text{and}\quad\lim_{n \to \infty} \frac{\mathcal{G}\,(u_n, \mathbb{B}^m)}{\mathcal{E}_{s, p} (u_n, \mathbb{B}^m)} = + \infty,
\]
then for every \(\varepsilon > 0\), there exists \(b_* \in \mathcal{N}\) and a measurable map \(u : \mathbb{R}^m \to \mathcal{N}\) such that \(\mathcal{E}_{s, p} (u, \mathbb{R}^m) \le \varepsilon\), \(u = b_*\) in \(\R^m \setminus \mathbb{B}^m_{1/2}\) and \(\mathcal{G}\,(u, \mathbb{B}^m) = + \infty\).
\end{description}
\end{theorem}

The proof of \cref{theoremCounterexample} follows the proof of the uniform boundedness for the weak-bounded approximation problem when \(s = 1\) \cite{Hang:2003}*{Theorem 9.6}.

\begin{proof}[Proof of \cref{theoremCounterexample} in the subcritical case \(sp<m\)]
\resetconstant
By passing if necessary to a subsequence, we can assume that for each \(n \in \N\), there exists a function $u_n: \mathbb{B}^m \to \mathcal{N}$ such that 
\[
 0< \mathcal{E}_{s,p} (u_n,\mathbb{B}^m)  \le \mu^{-n} \mathcal{G}\,(u_n,\mathbb{B}^m) < + \infty,
\]
where the parameter \(\mu > 1\) will be fixed later in the proof.
\medbreak

\noindent \emph{Step 1: Extension.}
By \cref{extensionLemma}, for each \(n \in \N\), there exists a function \(u_n^\mathrm{ext} : \mathbb{B}^m_3 \to \mathcal{N}\) such that 
\(u_n^\mathrm{ext} = u_n\) on \(\mathbb{B}^m\) and 
\[
 \mathcal{E}_{s, p} (u_n^\mathrm{ext}, \mathbb{B}^m_3) \le \Cl{unextc} \mathcal{E}_{s, p} (u_n, \mathbb{B}^m).
\]
In particular, we have, 
\begin{equation}
\label{estimateExtended}
  \mathcal{E}_{s, p} (u_n^\mathrm{ext}, \mathbb{B}^m_3) \le \Cr{unextc} \mathcal{E}_{s, p} (u_n, \mathbb{B}^m)
  \le \Cr{unextc} \mu^{-n} \mathcal{G}\,(u_n, \mathbb{B}^m) = \Cr{unextc} \mu^{-n} \mathcal{G}\,(u_n^\mathrm{ext}, \mathbb{B}^m).
\end{equation}

\medbreak

\noindent \emph{Step 2: Opening.}
We prove that we can make the map $u_n^\mathrm{ext}$ constant out of the ball $\mathbb{B}^n_2$. 
We take a Lipschitz-continous map \(\varphi : \mathbb{B}^m_6 \to \mathbb{B}^m_2\) such that \(\varphi (x)= x \) if \(\abs{x} \le 2\), and \(\varphi (x) = 0\) if \(\abs{x} \ge 3\). By \cref{openingLemma} with \(\rho = 5\), \(\eta = \frac{1}{5}\) and \(\lambda = \frac{3}{5}\), there exists a point \(a_n \in \mathbb{B}^m_1\) such that 
\[
 \mathcal{E}_{s, p} \bigl(u_n^\mathrm{ext}  \compose (\varphi ( \cdot - a_n) + a_n) , \mathbb{B}^m_5\bigr) \le \Cl{unopnc} \mathcal{E}_{s, p} (u_n^\mathrm{ext}, \mathbb{B}^m_3).
\]
By the conditions that we have imposed on \(\varphi\), if \(\abs{x} \le 1\) then \(\varphi (x - a_n) + a_n = x\) and 
\[
  \bigl(u_n^\mathrm{ext} \compose (\varphi ( \cdot - a_n) + a_n)\bigr) (x) = u_n^{\mathrm{ext}} (x) = u_n (x),
\]
whereas if \(\abs{x} \ge 4\), then \(\abs{x - a_n} \ge 3\) and \(\varphi (x - a_n) + a_n = a_n\) and thus 
\[
 \bigl(u_n^\mathrm{ext} \compose (\varphi ( \cdot - a_n) + a_n)\bigr) (x) = b_n := u_n^\mathrm{ext} (a_n).
\]
We define the map \(u_n^{\mathrm{opn}} : \R^m \to \mathcal{N}\) by
\[
u_n^\mathrm{\mathrm{opn}} (x)=
\begin{cases}
u_n^\mathrm{ext} \bigl(\varphi (x - a_n) + a_n)\bigr) & \text{if }x \in \mathbb{B}^m_4,\\
b_n &\text{if }x\in\R^m\setminus\mathbb{B}^m_4.
\end{cases}
\]
By construction, \(u_n^\mathrm{\mathrm{opn}} = u_n^\mathrm{ext} = u_n\) in \(\mathbb{B}^m\) and \(u_n^{\mathrm{opn}} = b_n\) in \(\R^m \setminus \mathbb{B}^m_4\). Moreover,
by \cref{lemmaBallToRN},
\[
 \mathcal{E}_{s, p} (u_n^{\mathrm{opn}}, \R^m) \le \Cl{itpxl} \mathcal{E}_{s, p} \bigl(u_n^\mathrm{ext} \compose  (\varphi ( \cdot - a_n) + a_n), \mathbb{B}^m_5\bigr).
\]
Finally, we have by \eqref{estimateExtended},
\begin{equation}
\label{IneqResultOpening}
 \mathcal{E}_{s, p} (u_n^\mathrm{\mathrm{opn}}, \R^m) \le \Cr{itpxl}\Cr{unopnc} \,\mathcal{E}_{s, p} (u_n^\mathrm{ext}, \mathbb{B}^m_3)
  \le \Cr{itpxl} \Cr{unopnc} \Cr{unextc} \mu^{-n} \,\mathcal{G}\,(u_n, \mathbb{B}^m) = \mu^{-n} \Cl{unopncb} \,\mathcal{G}\,(u_n^\mathrm{\mathrm{opn}}, \mathbb{B}^m),
\end{equation}
with \(\Cr{unopncb} =  \Cr{itpxl}\Cr{unopnc} \Cr{unextc}\).

\medbreak
\noindent \emph{Step 3: Clustering the maps.}
We fix a point \(b_* \in \mathcal{N}\).
Since the manifold $\mathcal{N}$ is connected, for each \(n \in \N\), there exists a smooth map \(v_n : \R^m \to \mathcal{N}\) such that \(v_n = b_n\) in \(\mathbb{B}^m_{1/2}\) and \(v_n = b_*\) in \(\mathbb{R}^m \setminus \mathbb{B}^m_{1}\).
We have by \cref{lemmaBallToRN} and by the smoothness of \(v_n\)
\[
 \mathcal{E}_{s,p}(v_n,\R^m) \le \Cl{rpxj} \bigl(\mathcal{E}_{s, p} (v_n, \mathbb{B}^m_2) + \mathcal{E}_{s, p} (v_n, \R^m \setminus \mathbb{B}^m_{1})\bigr)
 = \Cr{rpxj} \mathcal{E}_{s, p} (v_n, \mathbb{B}^m_2) < +\infty.
\]
The ball \(\mathbb{B}^m_{1/2}\) contains a cube \(Q\) of side-length \(1/\sqrt{m}\) that can be decomposed 
for each \(k \in \mathbb{N}_*\) into \(k^m\) cubes of side-length \(1/(k\sqrt{m})\). In particular, there is a set \(P_k \subset \mathbb{B}^m_{1/2}\) such that \(\# P_k = k^m\) and the balls \((\mathbb{B}^m_{1/(2k \sqrt{m})} (c))_{c \in P_k}\) are disjoint subsets of \(\mathbb{B}^m_{1/2}\).
We define for each \(c \in P_k\) the map 
\[
 v_{n, k, c} (x) = u_n^{\mathrm{opn}} \bigl(16 k\sqrt{m} (x - c) \bigr).
\]
and we observe that \(v_{n, k, c} (x) = b_n\) if \(x \in \R^m \setminus \mathbb{B}^m_{1/(4k \sqrt{m})} (c)\).
We define now
\[
  v_{n, k}(x) = 
  \begin{cases}
     v_{n,k,c} (x) & \text{if \(c \in P_k\) and \(x \in \mathbb{B}^m_{1/(2k \sqrt{m})} (c)\)},\\
     v_n (x) & \text{otherwise}.
  \end{cases}
\]
On the one hand, by an application of \cref{subadditive}, we have 
\[
\begin{split}
 \mathcal{E}_{s, p} (v_{n,k}, \R^m)
 &\le \Cl{pyxac} \Bigl(\mathcal{E}_{s,p} (v_n, \R^m) + \sum_{c \in P_c} \mathcal{E}_{s, p} (v_{n,k,c}, \R^m) \Bigr)\\
 &= \Cr{pyxac} \Bigl(\mathcal{E}_{s,p} (v_n, \R^m) + \frac{k^{sp}}{(16 \sqrt{m})^{m - sp}}  \mathcal{E}_{s, p} (u_n^{\mathrm{opn}}, \R^m) \Bigr).
\end{split}
\]
On the other hand, we have
\[
 \mathcal{E}_{s, p} (v_{n,k}, \R^m) \ge \sum_{c \in P_c} \mathcal{E}_{s, p} \Bigl(v_{n,k,c}, \mathbb{B}^m_{1/(16k \sqrt{m})} (c)\Bigr)
 = \frac{k^{sp}}{(16 \sqrt{m})^{m - sp}}  \mathcal{E}_{s, p} (u_n^{\mathrm{opn}}, \mathbb{B}^m).
\]
Since \(\mathcal{E}_{s, p} (u_n^\mathrm{\mathrm{opn}}, \mathbb{B}^m) = \mathcal{E}_{s, p} (u_n, \mathbb{B}^m) > 0\), this implies that we can choose \(k=k(n) \in \N_*\) in such a way that 
\[
 \nu \le \mathcal{E}_{s, p} (v_{n,k}, \R^m)
 \le 2 \Cr{pyxac} \frac{k^{sp}}{\bigl(16 \sqrt{m}\bigr)^{m - sp}} \mathcal{E}_{s, p} (u_n^{\mathrm{opn}}, \R^m),
\]
where the constant \(\nu > 0\) will be fixed at the end of the proof.
By superadditivity of the energy \(\mathcal{G}\), we have 
\[
  \mathcal{G}\,(v_{n,k}, \mathbb{B}^m) \ge \sum_{c \in P_c} \mathcal{G}\,\Bigl(v_{n,k,c}, \mathbb{B}^m_{1/(16 k \sqrt{m})} (c)\Bigr)
  = \frac{k^{sp}}{\bigl(16 \sqrt{m}\bigr)^{m - sp}} \mathcal{G}\,(u_n^{\mathrm{opn}}, \mathbb{B}^m).
\]
We have therefore by \eqref{IneqResultOpening},
\[
 \nu \le \mathcal{E}_{s, p} (v_{n,k}, \R^m) \le  \mu^{-n} 2\,\Cr{pyxac}\, \Cr{unopncb}\, \mathcal{G}\,(v_{n, k}, \mathbb{B}^m).
\]
We define the map \(u_n^{\mathrm{clstr}} : \R^m \to \mathcal{N}\) for every \(x \in \R^m\) by
\[
 u_n^{\mathrm{clstr}} (x) = v_{n, k} (x/\lambda),
\]
where \(\lambda \in (0, 1]\) is chosen by scaling in such a way that
\begin{equation}
\label{ineqResultClustering}
 \nu = \mathcal{E}_{s, p} ( u_n^{\mathrm{clstr}}, \R^m) \le \mu^{-n} \Cl{itxdpx}\, \mathcal{G}\,(u_n^{\mathrm{clstr}}, \mathbb{B}^m),
\end{equation}
with \(\Cr{itxdpx} = 2\, \Cr{pyxac} \,\Cr{unopncb}\). By construction, one has also \(u_n^{\mathrm{clstr}}=b_*\) out of \(\mathbb{B}^m\).

\medskip
\noindent\emph{Step 4: Gluing the maps.} 
If \(Q\) denotes a cube of side-length \(1/\sqrt{m}\) contained in \(\mathbb{B}^m_{1/2}\), by dyadic decomposition the cube \(Q\) contains a family of cubes of sidelengths \((2^{-n-1}/\sqrt{m})_{n \in \N}\) and thus, if we set \(\rho_n = 2^{-n - 2}/\sqrt{m}\), there exists a sequence of points \((a_n)_{n \in \N}\) such that the balls \(\bigl(\Bar{B}_{\rho_n} (a_n)\bigr)_{n \in \N}\) are disjoint balls contained in the open ball \(\mathbb{B}^m_{1/2}\) and the sequence \((a_n)_{n \in \N}\) converges to \(0\). 
We define the map \(u : \R^m \to \mathcal{N}\) for each point \(x \in \R^m\) by 
\[
 u (x) = 
 \begin{cases}
   u_n^{\mathrm{clstr}} \bigl(\frac{x - a_n}{\rho_n}\bigr) & \text{if \(x \in \mathbb{B}^m_{\rho_n} (a_n)\)},\\
   b_* & \text{otherwise}.
 \end{cases}
\]
If we take \(\mu = 2^{m - sp}\), we have by countable superadditivity (which is a consequence of finite superadditivity by the monotone convergence theorem for series), translation-invariance and scaling of the energy \(\mathcal{G}\), in view of \eqref{ineqResultClustering},
\[
 \mathcal{G}\,(u, \mathbb{B}^m) \ge \sum_{n \in \N} \mathcal{G}\,\bigl(u, \mathbb{B}^m_{\rho_n} (a_n)\bigr)
 = \sum_{n \in \N} \rho_n^{m - sp} \mathcal{G}\,(u_n^\mathrm{clstr}, \mathbb{B}^m)
 \ge \sum_{n \in \N} \frac{\nu\mu^n}{\Cr{itxdpx} \bigl(2^{n+2}\sqrt{m}\bigr)^{m - sp}} = + \infty.
\]
On the other hand, by choosing \(\nu > 0\) small enough, we have by \cref{subadditive} and by \eqref{ineqResultClustering} again
\[
 \mathcal{E}_{s, p} (u, \R^m)
 \le 2^p \sum_{n \in \N} \rho_n^{m - sp} \mathcal{E}_{s, p} (u_n^\mathrm{clstr}, \mathbb{R}^m)
 \le 2^p \sum_{n \in \N} \frac{\nu}{\bigl(2^{n+2}\sqrt{m}\bigr)^{m - sp}} \le \varepsilon < +\infty,
\]
since \(sp < m\).
\end{proof}

We now consider the critical case \(s = m/p\) and \(s < m\) (the last inequality excludes the case \(s=p=m=1\)). In this case,  the Sobolev energy is scaling invariant and it is not always possible to obtain a Sobolev map with finite energy by gluing an infinite number of rescaled copies of the \(u_n\). We use the assumption that \((\mathcal{E}_{m/p,p}(u_n,\mathbb{B}^m))_{n \in \N}\) goes to \(0\) to bypass this limitation and the following classical result:

\begin{lem}\label{nullCapacity}
Let \(m \in \N_\ast\), \(s \in (0, 1]\) and \(p \in (1,+\infty)\). If \(sp = m\), then there exists a sequence of maps \((w_n)_{n \in \N}\) in 
\(C^\infty_c (\mathbb{R}^m, [0,1])\) such that for each \(n\in \N\),  \(w_n = 1\) on \(\mathbb{B}^m\) and 
\[
 \lim_{n \to \infty} \mathcal{E}_{s, p} (w_n,\R^m) = 0.
\]
\end{lem}

The construction is classical and is related to the nonembedding of the critical Sobolev spaces into \(L^\infty\) and the null critical capacity of points. 
A direct way to construct such maps is to set \(w_n = w (\frac{1}{n}\ln \abs{x})\), where the function \(w \in C^\infty (\R, [0,1])\) satisfies \(w = 1\) on \((-\infty, 0]\) and \(w = 0\) on \([1, + \infty)\). When \(s = 1\) and \(p=m > 1\), the property follows by direct computation; when \(s \in (0, 1)\) the property follows from the fractional Sobolev embedding theorem.

\begin{proof}[Proof of \cref{theoremCounterexample} in the critical case \(sp=m\) and \(s < m\)]
\resetconstant
By passing if necessary to a subsequence, we can assume that there exists a sequence of measurable maps $u_n: \mathbb{B}^m \to \mathcal{N}$ such that 
\[
\lim_{n\to\infty}\mathcal{E}_{m/p,p}(u_n,\mathbb{B}^m)=0\quad\text{and}\quad 0< \mathcal{E}_{m/p,p} (u_n,\mathbb{B}^m)  \le 2^{-n} \mathcal{G}\,(u_n,\mathbb{B}^m) < + \infty.
\]
By Step 1 and Step 2 of the previous proof of \cref{theoremCounterexample} in the subcritical case \(sp < m\) (these steps do not use \(sp<m\)), we have existence of some maps \(u_n^\mathrm{\mathrm{opn}}:\R^m\to\mathcal{N}\) such that \(u_n^\mathrm{\mathrm{opn}}=u_n\) in \(\mathbb{B}^m\), \(u_n^\mathrm{\mathrm{opn}}=:b_n\in\mathcal{N}\) in \(\R^m\setminus\mathbb{B}^m_4\) and
\begin{equation}
\label{IneqResultOpeningCritical}
\mathcal{E}_{m/p,p}(u_n^\mathrm{\mathrm{opn}},\R^m)\le \Cl{fdjk} \mathcal{E}_{m/p,p}(u_n,\mathbb{B}^m)\le\Cr{fdjk}2^{-n}\mathcal{G}\,(u_n^\mathrm{\mathrm{opn}},\mathbb{B}^m).
\end{equation}

\medbreak
\noindent \emph{Step 3: Clustering the maps.}
Since the manifold $\mathcal{N}$ is connected, for each \(n \in \N\), there exists a Lipschitz-continuous curve \(\gamma_n:[0,1]\to\mathcal{N}\) such that \(\gamma_n(0)=b_\ast\) and \(\gamma_n(1)=b_n\).  We define for each \(\ell \in \N\), the mapping 
\(v_{n, \ell} =\gamma_n \compose w_\ell \colon \R^m\to\mathcal{N}\), where the map \(w_\ell\) is provided by \Cref{nullCapacity}.

By construction, we have \(v_{n, \ell}(x)=b_n\) on \(\mathbb{B}^m\) and \(v_{n, \ell}(x)=b_*\) on \(\R^m \setminus \mathbb{B}^m_{R_\ell}\) for some \(R_\ell \in (1, + \infty)\). We take \(k\in\N_*\) and pick a family of \(k\) disjoint balls \(\mathbb{B}^m_{\rho_1}(c_1),\dots,\mathbb{B}^m_{\rho_k}(c_k)\) in \(\mathbb{B}^m_{1/2}\), with \(c_1,\dots,c_k\in\mathbb{B}^m_{1/2}\) and \(\rho_1,\dots,\rho_k>0\). We define for each \(i\in\{1,\dots,k\}\) the map 
\[
 v_{n, k, i} (x) = u_n^{\mathrm{opn}} \Bigl(\frac{8}{\rho_i} (x - c_i) \Bigr).
\]
and we observe that \(v_{n, k, i} (x) = b_n\) if \(x \in \R^m \setminus \mathbb{B}^m_{\rho_i/2} (c_i)\).
We define now
\[
  v_{n, k,\ell}(x) = 
  \begin{cases}
     v_{n,k,i} (x) & \text{if \(i\in\{1,\dots,k\}\) and \(x \in \mathbb{B}^m_{\rho_i}(c_i)\)},\\
     v_{n, \ell} (x) & \text{otherwise}.
  \end{cases}
\]
On the one hand, by an application of \cref{subadditive} and by scaling invariance, we have 
\[
\begin{split}
 \mathcal{E}_{m/p, p} (v_{n, k, \ell}, \R^m)
 &\le 2^p \Bigl(\mathcal{E}_{m/p,p} (v_{n, \ell}, \R^m) + \sum_{i=1}^k \mathcal{E}_{m/p, p} (v_{n,k,i}, \R^m) \Bigr)\\
 &\le \Cl{kloi} \Bigl(\mathrm{Lip}(\gamma_n)^p\mathcal{E}_{m/p, p} (w_\ell, \R^m) + k \mathcal{E}_{m/p, p} (u_n^{\mathrm{opn}}, \R^m) \Bigr).
\end{split}
\]
On the other hand, by superadditivity of \(\mathcal{E}_{m/p,p}\) and by an application of \cref{lemmaBallToRN}, we have
\[
 \mathcal{E}_{m/p, p} (v_{n, k, \ell}, \R^m) \ge \sum_{i=1}^k \mathcal{E}_{m/p, p} \bigl(v_{n,k,i}, \mathbb{B}^m_{\rho_i} (c_i)\bigr)
 = k  \mathcal{E}_{m/p, p} (u_n^{\mathrm{opn}}, \mathbb{B}^m_8)\ge \Cl{fdsfds} k \mathcal{E}_{m/p, p} (u_n^{\mathrm{opn}}, \R^m).
\]
Since \(0<\mathcal{E}_{m/p, p} (u_n^\mathrm{\mathrm{opn}}, \R^m) \to 0\) as \(n\to\infty\) and \(\mathcal{E}_{m/p,p}(w_\ell,\R^m)\to 0\) as \(\ell\to\infty\), by passing to a subsequence if necessary, one can assume that there exist \(k=k(n) \in \N_*\) and \(\ell = \ell(n)\in\N_*\) such that
\[
\mathrm{Lip}(\gamma_n)^p\mathcal{E}_{m/p, p} (w_\ell, \R^m) \le 2^{-n}\nu \le k\mathcal{E}_{m/p,p}(u_n^{\mathrm{opn}},\R^m)\le 2^{-n+1}\nu,
\]
where \(\nu > 0\) is a constant to be fixed at the end of the proof.
In particular, by scaling invariance, the map \(u_n^\mathrm{clstr}:=v_{n,k(n),\ell (n)}(R_\ell \cdot)\) satisfies
\[
\mathcal{E}_{m/p,p}(u_n^\mathrm{clstr},\R^m) \le \Cr{kloi}2^{-n+2}\nu.
\]
By superadditivity and scaling invariance of the energy \(\mathcal{G}\), and by \eqref{IneqResultOpeningCritical}, we have furthermore
\begin{equation}
\label{ineqResultClusterCritical}
  \mathcal{G}\,(u_n^{\mathrm{clstr}}, \mathbb{B}^m) \ge k \mathcal{G}\,(u_n^{\mathrm{opn}}, \mathbb{B}^m)\ge k 2^n \Cr{fdjk}^{-1}\mathcal{E}_{m/p,p}(u_n^\mathrm{\mathrm{opn}},\R^m)\ge \nu \Cr{fdjk}^{-1}>0.
\end{equation}
By construction, we have also \(u_n^{\mathrm{clstr}}=b_*\) in \(\R^m \setminus \mathbb{B}^m\).

\medskip
\noindent\emph{Step 4: Gluing the maps.} 
There exist a sequence of points \((a_n)_{n \in \N}\subset\mathbb{B}^m_{1/2}\) and a sequence of radii \((\rho_n)_{n\in\N}\) in \((0,+\infty)\) such that the balls \(\bigl(\Bar{B}_{\rho_n} (a_n)\bigr)_{n \in \N}\) are disjoint balls contained in the open ball \(\mathbb{B}^m_{1/2}\) and the sequence \((a_n)_{n \in \N}\) converges to \(0\). 
We define the map \(u : \R^m \to \mathcal{N}\) for each \(x \in \R^m\) by 
\[
 u (x) = 
 \begin{cases}
   u_n^{\mathrm{clstr}} \bigl(\frac{x - a_n}{\rho_n}\bigr) & \text{if \(x \in \mathbb{B}^m_{\rho_n} (a_n)\)},\\
   b_* & \text{otherwise}.
 \end{cases}
\]
We have by superadditivity, translation invariance and scaling invariance of the energy \(\mathcal{G}\), in view of \eqref{ineqResultClusterCritical},
\[
 \mathcal{G}\,(u, \mathbb{B}^m) \ge \sum_{n \in \N} \mathcal{G}\,\bigl(u, \mathbb{B}^m_{\rho_n} (a_n)\bigr)
 = \sum_{n \in \N} \mathcal{G}\,(u_n^\mathrm{clstr}, \mathbb{B}^m)=+\infty.
\]
On the other hand, we have by \cref{subadditive}, if \(\nu > 0\) is small enough,
\[
 \mathcal{E}_{m/p, p} (u, \R^m)
 \le 2^p \sum_{n \in \N} \mathcal{E}_{m/p, p} (u_n^\mathrm{clstr}, \mathbb{R}^m)
 \le 2^p \Cr{kloi}\nu \sum_{n \in \N} 2^{-n+2} \le \varepsilon < + \infty.
\]
Since we have also \(u=b_*\) out of \(\mathbb{B}^m_{1/2}\), this concludes the proof in the critical case.
\end{proof}

We finally consider the case where \(sp>m\) or \(s = m = p = 1\).

\begin{proof}[Proof of \cref{theoremCounterexample} in the supercritical case \(sp>m\) or \(s = m = p = 1\)]
\resetconstant
By passing if necessary to a subsequence, we can assume that there exists a sequence of measurable maps $u_n: \mathbb{B}^m \to \mathcal{N}$ such that 
\[
\lim_{n\to\infty}\mathcal{E}_{s,p}(u_n,\mathbb{B}^m)=0\quad\text{and}\quad 0< \mathcal{E}_{s,p} (u_n,\mathbb{B}^m)  \le \mu^{-n} \mathcal{G}\,(u_n,\mathbb{B}^m) < + \infty,
\]
with \(\mu > 1\) that will be determined at the end of the proof.
By Step 1 and Step 2 of the proof of \cref{theoremCounterexample} in the subcritical case, we have existence of some maps \(u_n^\mathrm{\mathrm{opn}}:\R^m\to\mathcal{N}\) such that \(u_n^\mathrm{\mathrm{opn}}=u_n\) in \(\mathbb{B}^m\), \(u_n^\mathrm{\mathrm{opn}}=:b_n\in\mathcal{N}\) in \(\R^m\setminus\mathbb{B}^m_4\) and
\begin{equation}
\label{IneqResultOpeningSuperCritical}
\mathcal{E}_{s,p}(u_n^\mathrm{\mathrm{opn}},\R^m)\le \Cl{fdjke} \mathcal{E}_{s,p}(u_n,\mathbb{B}^m)\le\Cr{fdjke}\mu^{-n}\mathcal{G}\,(u_n^\mathrm{\mathrm{opn}},\mathbb{B}^m).
\end{equation}

\medbreak
\noindent \emph{Step 3: Fixing the boundary value.}
Since the manifold \(\mathcal{N}\) is compact, by passing if necessary to a subsequence, one can assume that the sequence \((b_n)_{n\in\N}\) converges to some point \(b_*\in\mathcal{N}\) as \(n\to\infty\). 
We consider a function \(w_* \in C^1_c (\R^n,[0,1])\) such that \(w_* = 0\) in \(\R^m \setminus \mathbb{B}^m_1\) and \(w_* = 1\) on \(\mathbb{B}^m_{1/2}\). 
Since $\mathcal{N}$ is connected, for each \(n \in \N\), there exists a Lipschitz-continuous curve \(\gamma_n : [0, 1] \to \R\) such that \(\gamma_n (0) = b_*\), \(\gamma_n (1) = b_n\) and \(\operatorname{Lip} (\gamma_n) \le 2 d_{\mathcal{N}} (b_n ,b_*)\). Then the map \(v_n=\gamma_n\compose w_\ast\) satisfies
\[
 \mathcal{E}_{s,p}(v_n,\R^m) \le \operatorname{Lip} (\gamma_n)^p \mathcal{E}_{s, p} (w_*) \le \C d_\mathcal{N}(b_n,b_*)^p .
\]
If \(sp > m\), for every \(\rho\in (0,\frac{1}{16})\), we define now
\[
  v_{n, \rho}(x) = 
  \begin{cases}
     u_n^{\mathrm{opn}} \bigl(\frac{x}{\rho}\bigr) & \text{if }x\in\mathbb{B}^m_{1/2},\\
     v_n (x) & \text{if }x\in\R^m\setminus\mathbb{B}^m_{1/2}.
  \end{cases}
\]
By an application of \cref{subadditive}, we have 
\[
\begin{split}
 \mathcal{E}_{s, p} (v_{n,\rho}, \R^m)
 &\le 2^p \Bigl(\mathcal{E}_{s,p} (v_n, \R^m) + \mathcal{E}_{s, p} (u_n^{\mathrm{opn}}(\cdot/\rho), \R^m) \Bigr)\\
 &\le \Cl{kloia} \Bigl(d_\mathcal{N}(b_n,b_*)^p + \frac{1}{\rho^{sp-m}} \mathcal{E}_{s, p} (u_n^{\mathrm{opn}}, \R^m) \Bigr).
\end{split}
\]
Since \(0<\mathcal{E}_{s, p} (u_n^\mathrm{\mathrm{opn}}, \R^m) \to 0\) and \(d_\mathcal{N}(b_n,b_*)\to 0\) as \(n\to\infty\), by passing to a subsequence if necessary, one can assume that there exists \(\rho=\rho(n) \in (0,\frac{1}{16})\) such that
\[
 \mathcal{E}_{s, p} (v_{n,\rho}, \R^m)
\le 2\Cr{kloia}\frac{1}{\rho^{sp-m}} \mathcal{E}_{s, p} (u_n^{\mathrm{opn}}, \R^m)\le \nu \mu^{-n},
\]
where \(\nu >0\) is a constant whose value will be fixed at the end of the proof.
Moreover, by scaling of the energy \(\mathcal{G}\) and by \eqref{IneqResultOpeningSuperCritical}, we have 
\[
\mathcal{G}\,(v_{n,\rho}, \mathbb{B}^m) \ge\mathcal{G}\,\big(u_n^\mathrm{opn}({\cdot}/{\rho}),\mathbb{B}^m_\rho\big)= \frac{1}{\rho^{sp-m}} \mathcal{G}\,(u_n^\mathrm{opn}, \mathbb{B}^m)\ge 
\frac{\Cr{fdjke}^{-1}\mu^n}{\rho^{sp-m}} \mathcal{E}_{s,p} (u_n^\mathrm{opn},\R^m).
\]
We have therefore
\[
 \mathcal{E}_{s, p} (v_{n,\rho}, \R^m)\le \nu \mu^{-n}\quad\text{and}\quad\mathcal{E}_{s, p} (v_{n,\rho}, \R^m) \le\Cr{unopncb}\mu^{-n} \mathcal{G}\,(v_{n, \rho}, \mathbb{B}^m).
\]
We define the map \(u_n^{b_*} : \R^m \to \mathcal{N}\) for every \(x \in \R^m\) by
\[
u_n^{b_*} (x) = v_{n,\rho} (x/\lambda),
\]
where \(\lambda \in (0, 1]\) is chosen by scaling in such a way that
\begin{equation}
\label{ineqResultClusteringSuper}
\nu \mu^{-n} = \mathcal{E}_{s, p} ( u_n^{b_*}, \R^m) \le  \Cl{itxdpxx}\,\mu^{-n} \mathcal{G}\,(u_n^{b_*}, \mathbb{B}^m),
\end{equation}
By construction, we have also \(u_n^{b_*}=b_*\) out of \(\mathbb{B}^m\).

If \(s = p = m = 1\), we proceed as in the proof of \cref{theoremCounterexample} when \(sp = m\), with \(w_*\) instead of \(w_\ell\), relying on the smallness of \(\operatorname{Lip} (\gamma_n)\) instead of the smallness of the energy \(\mathcal{E}_{s, p} (w_\ell)\).

\medskip
\noindent\emph{Step 4: Gluing the maps.} 
There exists a sequence of points \((a_n)_{n \in \N}\) such that the balls \(\bigl(\Bar{B}_{\rho_n} (a_n)\bigr)_{n \in \N}\) with \(\rho_n = 2^{-n - 2}/\sqrt{m}\) are disjoint balls contained in the open ball \(\mathbb{B}^m_{1/2}\) and the sequence of points \((a_n)_{n \in \N}\) converges to \(0\). 
The map \(u : \R^m \to \mathcal{N}\) is defined at each point \(x \in \R^m\) by 
\[
u (x) = 
\begin{cases}
u_n^{b_*} \bigl(\frac{x - a_n}{\rho_n}\bigr) & \text{if \(x \in \mathbb{B}^m_{\rho_n} (a_n)\)},\\
b_* & \text{otherwise}.
\end{cases}
\]
If we take \(\mu>2^{sp-m}\), we have by countable superadditivity, translation-invariance and scaling of the energy \(\mathcal{G}\), in view of \eqref{ineqResultClusteringSuper},
\[
\mathcal{G}\,(u, \mathbb{B}^m) \ge \sum_{n \in \N} \mathcal{G}\,\bigl(u, \mathbb{B}^m_{\rho_n} (a_n)\bigr)
= \sum_{n \in \N} \frac{\mathcal{G}\,\bigl(u_n^{b_*}, \mathbb{B}^m\bigr)}{\rho_n^{sp-m}}
\ge \sum_{n \in \N} \frac{\nu(2^{n+2}\sqrt{m})^{sp-m}}{\Cr{itxdpxx}} = + \infty.
\]
On the other hand, we have by \cref{subadditive} and by the inequality \eqref{ineqResultClusteringSuper} again
\[
\mathcal{E}_{s, p} (u, \R^m)
\le 2^p\sum_{n \in \N} \frac{\mathcal{E}_{s, p} (u_n^{b_*}, \mathbb{R}^m)}{\rho_n^{sp-m}}
= 2^p \nu \sum_{n \in \N} \frac{\bigl(2^{n+2}\sqrt{m}\bigr)^{sp-m}}{\mu^n} \le \varepsilon < + \infty,
\]
if \(\nu > 0\) is small enough, 
since \(sp \ge m\).
\end{proof}

\subsection{Density of counterexamples}
We use now \cref{theoremCounterexample} and ingredients of its proof to prove that when \(sp \le m\), Sobolev maps with infinite energy \(\mathcal{G}\) are dense. 

\begin{theorem}[Density of counterexamples]
\label{theoremDenseCounterexamples}
  Let \(m \in \N_*\), \(s \in (0, 1]\), \(p \in [1, +\infty)\), \(\mathcal{N}\) be a connected Riemannian manifold, and let \(\mathcal{G}\) be an  energy over $\R^m$ with state space $\mathcal{N}$. 
  Assume that for every measurable map \(u:\R^m\to\mathcal{N}\)
  \begin{enumerate}[(i)]
    \item (superadditivity) for all open sets \(A,B\subset\R^m\) with disjoint closure,
    \[
    \mathcal{G}\,(u,A\cup B)\geq \mathcal{G}\,(u,A)+\mathcal{G}\,(u,B),
    \]
    \item (scaling) for all $\lambda>0$, $h\in\R^m$ and any open set $A\subset\R^m$, one has 
    \[
    \mathcal{G}\,(u ,h+\lambda A)=\lambda^{m-sp} \mathcal{G}\,(u(h + \lambda \cdot),A).
    \]
  \end{enumerate}
    Assume furthermore that \(sp \le m\), \(s < m\) and that there exists a sequence \((u_n)_{n \in \N}\) of measurable maps \(u_n:\mathbb{B}^m \to \mathcal{N}\) such that for each \(n \in \N\), \(\mathcal{E}_{s, p} (u_n, \mathbb{B}^m) > 0\), \(\mathcal{G}\,(u_n,\mathbb{B}^m)<+\infty\), and 
    \[
    \lim_{n \to \infty} \frac{\mathcal{G}\,(u_n, \mathbb{B}^m)}{\mathcal{E}_{s, p} (u_n, \mathbb{B}^m)} = + \infty,\quad\text{with}\quad \lim_{n\to\infty}\mathcal{E}_{s,p}(u_n,\mathbb{B}^m)=0 \quad\text{if }sp=m.
    \]
Then, for every \(\varepsilon > 0\) and if the map \(v : \mathbb{B}^m \to \mathcal{N}\) is measurable and \(\mathcal{E}_{s, p} (v,\mathbb{B}^m) < + \infty\), there exists a measurable map \(u : \mathbb{B}^m \to \mathcal{N}\) such that 
\begin{enumerate}[(i)]
 \item \(u = v\) on \(\mathbb{B}^m \setminus \mathbb{B}^m_\varepsilon\),
 \item \(\mathcal{E}_{s, p} (u,\mathbb{B}^m) \le \mathcal{E}_{s, p} (v,\mathbb{B}^m) + \varepsilon\),
 \item \(\mathcal{G}\,(u,\mathbb{B}^m) = + \infty\).
\end{enumerate}
\end{theorem}

\Cref{theoremDenseCounterexamples} implies that there exists a sequence \((v_n)_{n \in \N}\) such that \(v_n = v\) on \(\mathbb{B}^m \setminus \mathbb{B}^m_{1/n}\) and \(\limsup_{n \to \infty} \mathcal{E}_{s, p} (v_n) \le \mathcal{E}_{s, p} (v)\), which implies in particular that the sequence \((v_n)_{n \in \N}\) converges strongly to \(v\) in \(W^{s, p}(\mathbb{B}^m, \mathcal{N})\).

\begin{proof}[Proof of \cref{theoremDenseCounterexamples}]
\resetconstant
We proceed in two steps: we first open the map \(v\) by making it constant in a neighbourhood of \(0\) and then we insert a singularity of \cref{theoremCounterexample}.

\medbreak

\noindent\emph{Step 1: Opening.}
We choose a Lipschitz-continuous function \(\varphi : \R^m \to \R^m\) such that \(\varphi (x) = 0\) if \(\abs{x} \le 3/10\) and \(\varphi (x) = x\) if \(\abs{x} \ge 1/2\).
Given \(\delta \in (0,1)\), we define \(\varphi_\delta : \mathbb{B}^m_{9\delta/10} \to \mathbb{B}^m_{9\delta /10}\) by \(\varphi_\delta (x) = \delta \varphi (x/\delta)\).
We apply \cref{openingLemma} with \(\rho = 4\delta/5\), \(\eta = 1/8\) and \(\lambda = 5/4\), and we obtain the existence of a point \(a \in \mathbb{B}^m_{\delta/10}\) such that
\begin{equation}
\label{est_open}
 \mathcal{E}_{s, p} \bigl(v \compose (\varphi_\delta (\cdot - a) + a), \mathbb{B}^m_{4\delta/5}\bigr) \le \C \mathrm{Lip}(\varphi)^{sp}\, \mathcal{E}_{s, p} (v, \mathbb{B}^m_\delta),
 \end{equation}
since \(\mathrm{Lip}(\varphi_\delta) = \mathrm{Lip} (\varphi)\).
We observe that for \(x\in \mathbb{B}^m_{4\delta/5}\),
\[
\varphi_\delta (x - a)+a=
\begin{cases}
x&\text{if \(|x|\geq 3\delta/5\) \ \ (since then \(|x-a|\geq 3\delta/5-|a|\geq \delta/2\))},\\
a&\text{if \(|x|\leq \delta/5\)  \ \ (since then \(|x-a|\leq \delta/5+|a|\leq 3\delta/10\)).}
\end{cases}
\]

\medskip
\noindent
\emph{Step 2: Inserting the singularity.}
Since \(sp \le m\) and \(s < m\), we apply \cref{theoremCounterexample} in the critical or subcritical case, with \(b_* = v(a)\) and we obtain a map \(w : \R^m \to \mathcal{N}\) such that 
\(w = b_*\) on \(\R^m \setminus \mathbb{B}^m\), \(\mathcal{G}\,(w, \mathbb{B}^m) = + \infty\) and \(\mathcal{E}_{s, p} (w, \mathbb{R}^m) \le \xi\), where \(\xi>0\) will be fixed at the end of the proof.
We define the map \(u : \mathbb{B}^m \to \mathcal{N}\) for \(x \in \mathbb{B}^m\) by
\[
 u (x)
 = \begin{cases}
     v(x) & \text{if \(\abs{x} \ge 4\delta/5\)},\\
     v (\varphi_\delta (x - a) + a) & \text{if \(\delta/5 \le \abs{x} \le 4\delta/5\)},\\
     w (10x/\delta) & \text{if \(\abs{x} \le \delta/5\)}.
   \end{cases}
\]
By a double application of \cref{lemmaBallToRN}, we have for every \(\sigma \in (\delta,1)\),
\[
\mathcal{E}_{s,p}(u,\mathbb{B}^m_\sigma)\le \C \bigl(\mathcal{E}_{s,p}(u,\mathbb{B}^m_\sigma\setminus\mathbb{B}^m_{3\delta/5})+\mathcal{E}_{s,p}(u,\mathbb{B}^m_{4\delta/5}\setminus\mathbb{B}^m_{\delta/10})+\mathcal{E}_{s,p}(u,\mathbb{B}^m_{\delta/5})\bigr).
\]
Since \(u=v\) on \(\mathbb{B}^m_\sigma\setminus\mathbb{B}^m_{3\delta/5}\), \(u=v(\varphi_\delta(\cdot-a)+a)\) on \(\mathbb{B}^m_{4\delta/5}\setminus\mathbb{B}^m_{\delta/10}\) and \(u=w(10x/\delta)\) on \(\mathbb{B}^m_{\delta/5}\), and since \(\sigma>\delta\), this implies by \eqref{est_open}
\[
 \mathcal{E}_{s,p} (u, \mathbb{B}^m_\sigma)\le \Cl{gluinter} \bigl( \mathcal{E}_{s,p} (v , \mathbb{B}^m_\sigma) + \mathcal{E}_{s,p} (w, \R^m)\bigr)\le  \Cr{gluinter} \bigl( \mathcal{E}_{s,p} (v , \mathbb{B}^m_\sigma) + \xi\bigr).
\]
We assume now that \(\sigma \ge 2 \delta\), and we apply \cref{lemmaBallToRN} with \(\rho=\sigma\) and \(\eta=\delta/\sigma\leq 1/2\). We obtain
\[
\begin{split}
 \mathcal{E}_{s,p} (u, \mathbb{B}^m) 
& \le  \C \mathcal{E}_{s,p} (u,\mathbb{B}^m_\sigma)+(1+\C\eta^m) \mathcal{E}_{s,p} (u,\mathbb{B}^m\setminus\mathbb{B}^m_\delta)
 \\
& \le 
\mathcal{E}_{s,p} (v, \mathbb{B}^m) 
+\Cl{jrdn}\Bigl(
 \Big(\frac{\delta}{\sigma}\Big)^m \mathcal{E}_{s,p} (v, \mathbb{B}^m) + \mathcal{E}_{s,p} (v , \mathbb{B}^m_\sigma) +  \xi
\Bigr).
\end{split}
\]
In order to obtain the conclusion, we first fix \(\sigma \in (0, 1)\) such that 
\[
 \mathcal{E}_{s,p} (v , \mathbb{B}^m_\sigma) \le \frac{\varepsilon}{3\Cr{jrdn}},
\]
next \(\delta \in (0, 1)\) such that \(\delta \le \frac{\sigma}{2}\), \(\delta \le \varepsilon\) and 
\[
 \Big(\frac{\delta}{\sigma}\Big)^m \mathcal{E}_{s,p} (v, \mathbb{B}^m) \le \frac{\varepsilon}{3\Cr{jrdn}},
 \]
this allows us then to construct the points \(a\in\mathbb{B}^m_{\delta/10}\) and \(b_*=v(a)\) and the obstruction \(w\) with \(\xi = \frac{\varepsilon}{3\Cr{jrdn}}\).
\end{proof}

\section{Concrete uniform boundedness principles}
\label{concreteUBP}

\subsection{Extension of traces}
We apply \cref{theoremCounterexample} to prove a uniform boundedness principle for the extension problem (\cref{extensionthm}).

\begin{proof}[Proof of \cref{extensionthm}]
Let \(m \in \N_\ast\), \(s \in (0, 1]\) and \(p \in [1, +\infty)\) and assume by contradiction that the linear bound does not hold. 
Then by \cref{theoremCounterexample} with \(\mathcal{G} = \mathcal{E}_{r,q}^{\mathrm{ext}}\) there exist a map \(u \in W^{s, p} (\mathbb{R}^m, \mathcal{N})\) and \(b_\ast\in\mathcal{N}\) such that \(\mathcal{E}_{r, q}^{\mathrm{ext}} (u, \mathbb{B}^m) = + \infty\) and \(u = b_*\) in \(\R^m \setminus \mathbb{B}^m\). 
If \(\mathcal{M} = \R^m\) we have a contradiction. Otherwise, \(\mathcal{M}\) is a compact Riemannian manifold, for which we consider a local chart \(\Phi : \mathbb{B}^m_2 \to \mathcal{M}\). We define the map \(\Tilde{u} : \mathcal{M} \to \mathcal{N} \) by 
\[
  \Tilde{u} (x)
  =
  \begin{cases}
    u \bigl(\Phi^{-1} (x)\bigr) & \text{if \(x \in \Phi (\mathbb{B}^m_1)\)},\\
    b_* & \text{otherwise}.
  \end{cases}
\]
Since \(\mathcal{M}\) is compact, we conclude by a counterpart of the gluing technique of \cref{lemmaBallToRN}. 
\end{proof}	

\begin{theorem}
\label{theoremExtensionDense}%
Let \(s, r \in (0, 1]\), \(p, q \in [1, +\infty)\), \(m \in \N_*\), \(\mathcal{M}\) be a Euclidean space or a compact Riemannian manifold of dimension \(m\) and \(\mathcal{N}\) be a connected Riemannian manifold.
If \(sp = rq - 1 \le m\) and \(s < m\) and if every map in a nonempty open subset of \(W^{s, p} (\mathcal{M}, \mathcal{N})\) is the trace of some map in \(W^{r, q} (\mathcal{M} \times (0, + \infty), \mathcal{N})\), then there exists a constant $C>0$ such that for each measurable function 
$u: \mathbb{B}^m \to \mathcal{N}$ such that, if either \(sp <m\) or \(\mathcal{E}_{s, p} (u, \mathbb{B}^m) \le 1/C\), then 
\[
  \mathcal{E}^{\mathrm{ext}}_{r ,q}(u,\mathbb{B}^m)\leq C \, \mathcal{E}_{s,p}(u,\mathbb{B}^m).
\]
\end{theorem}
\begin{proof}[Proof of \cref{theoremExtensionDense}]
\resetconstant
We assume by contradiction that the estimate does not hold.
Let \(v \in W^{s, p} (\mathcal{M}, \mathcal{N})\) and let \(\Phi : \mathbb{B}^m_2 \to \mathcal{M}\) be a local chart.
We apply \Cref{theoremDenseCounterexamples} to the map \(v \compose \Phi\) with the energy \(\mathcal{E}^{\mathrm{ext}}_{r,q}\), and we obtain a sequence of maps \((u_n)_{n \in \N_*}\) such that 
\(u_n = v \compose \Phi\) in \(\mathbb{B}^m \setminus \mathbb{B}^m_{1/n}\), \(\mathcal{E}^{\mathrm{ext}}_{r, q} (u_n, \mathbb{B}^m)= + \infty\) and \(\limsup_{n \to \infty} \mathcal{E}^{\mathrm{ext}}_{r, q} (u_n, \mathbb{B}^m)\le \mathcal{E}^{\mathrm{ext}}_{r, q} (v \compose \Phi, \mathbb{B}^m)\). We define now \(v_n:\mathcal{M}\to\mathcal{N}\) by
\[
 v_n(x)
 =\begin{cases}
    u_n \bigl(\Phi^{-1} (x)\bigr) & \text{if \(x \in \Phi (\mathbb{B}^m)\)},\\
    v(x) & \text{otherwise}.
  \end{cases}
\]
Since \(v_n = v\) in \(\Phi (\mathbb{B}^m \setminus \mathbb{B}^m_{1/n})\), we deduce by a counterpart of \cref{lemmaBallToRN} that
\[
\limsup_{n \to \infty} \mathcal{E}_{s, p} (v_n, \mathcal{M}) \le \mathcal{E}_{s, p} (v, \mathcal{M}) 
\]
and thus the sequence \((v_n)_{n \in \N_\ast}\) converges strongly to \(v\) in \(W^{s, p} (\mathcal{M}, \mathcal{N})\) but for each \(n \in \mathbb{N}\), \(\mathcal{E}^{\mathrm{ext}}_{r, q} (v_n, \mathcal{M})= + \infty\), which contradicts the assumption.
\end{proof}

In view of the estimate \eqref{eqBethuelGrowth} of Bethuel \cite{Bethuel2014Extension}*{(1.36)}, \cref{theoremExtensionDense} implies that if \(\mathcal{N}\) is compact, if $sp=p-1<\mathrm{dim}(\mathcal{M})$ and if either \(\pi_1 (\mathcal{N})\) is infinite or \(\pi_{j}(\mathcal{N}) \not \simeq \{0\}\) for some \(j \in \{2, \dotsc, \floor{p} - 1\}\), then the set of maps in \(W^{1-1/p,p} (\mathcal{M}, \mathcal{N})\) that are not traces of maps in \(W^{1, p}(\mathcal{M} \times \R_+, \mathcal{N})\) is dense.

\subsection{Weak-bounded approximation}
The proof of \cref{approximationthm} is similar to the proof of \cref{extensionthm}, with \(\mathcal{G} = \mathcal{E}^{\mathrm{rel}}_{r, q}\).
The counterpart of \cref{theoremExtensionDense} is 

\begin{theorem}
\label{theoremWkBddApproxDense}
Let $s, r\in (0, 1]$, $p, q \in [1, +\infty)$, $m \in \N_*$, \(\mathcal{M}\) be a Euclidean space or a compact Riemannian manifold of dimension \(m\) and let \(\mathcal{N}\) be a connected Riemannian manifold.
If $sp = rq <m$ and if every map in a nonempty open set of \(W^{s, p} (\mathcal{M},\mathcal{N})\) has a weak-bounded approximation in \(W^{r, q} (\mathcal{M},\mathcal{N})\), then there exists a constant $C>0$ such that for each measurable function $u:\mathbb{B}^m \to \mathcal{N}$, one has
\[
  \mathcal{E}^{\mathrm{rel}}_{r,q}(u,\mathbb{B}^m)
  \leq C \,\mathcal{E}_{s,p}(u,\mathbb{B}^m).
\] 
\end{theorem}

In view of the failure of a linear bound for the weak-bounded approximation problem in the space \(W^{1, 3} (\mathbb{B}^m, \mathbb{S}^2)\) when \(m \ge 4\) \cite{Bethuel2014Weak}, we obtain as a consequence of \cref{theoremWkBddApproxDense} the density of mappings that have no weak-bounded approximation in \(W^{1, 3} (\mathcal{M}, \mathbb{S}^2 )\) when \(\dim \mathcal{M}\ge 4\).

\subsection{Lifting problem}
\Cref{liftingthm} is also proved as  \cref{extensionthm}, with \(\mathcal{G} = \mathcal{E}^{\mathrm{lift}}_{r, q}\).
The counterpart of \cref{theoremExtensionDense,theoremWkBddApproxDense} is 

\begin{theorem}
\label{theoremLiftingDense}%
Let $s,r\in (0, 1]$, $p,q\in [1, +\infty)$, $m \in \N_*$, \(\mathcal{M}\) be a Euclidean space or a compact Riemannian manifold of dimension \(m\), \(\mathcal{N}\) and \(\mathcal{F}\) be Riemannian manifold manifolds with \(\mathcal{N}\) connected and \(\pi : \mathcal{F} \to \mathcal{N}\). If \(rq = sp \le m\), \(s < m\) and if for every map \(u\) in a nonempty open subset of $W^{s,p}(\mathcal{M},\mathcal{N})$ there exists \(\varphi \in W^{r, q} (\mathcal{M}, \mathcal{F})\) such that \(\pi \compose \varphi = u\), then there exists a constant $C>0$ such that for each measurable function \(u:\mathbb{B}^m\to \mathcal{N}\), if either \(sp < m\) or  \(\mathcal{E}_{s,p}(u,\mathbb{B}^m) \le 1/C\),
\[
  \mathcal{E}^\mathrm{lift}_{r,q}(u,\mathbb{B}^m)\leq C\, \mathcal{E}_{s,p}(u,\mathbb{B}^m).
\]
\end{theorem}

In view of the estimate \eqref{nonestimateMMM} of Merlet \cite{Merlet2006}*{Theorem 1.1} and of Mironescu and Molnar \cite{MironescuMolnar2015}*{Proposition 5.7}, \cref{theoremLiftingDense} implies that maps in \(W^{s, p}(\mathcal{M}, \mathbb{S}^1)\) having no lifting in \(W^{s, p}(\mathcal{M}, \R)\) are dense when \(s \in (0, 1)\) and \(1 < sp < \dim \mathcal{M}\).

When \(sp > 2\) and \(\mathcal{M}\) is simply--connected, mappings in \(W^{s, p}(\mathcal{M}, \mathbb{S}^1)\)
still have a lifting in the larger space \(W^{s, p}(\mathcal{M}, \R)+W^{1,sp}(\mathcal{M}, \R)\) \citelist{\cite{BourgainBrezisMironescu2002}*{Theorem 3}\cite{BourgainBrezis2003}*{Theorem 4}\cite{BourgainBrezisMironescu2004}*{Theorem 3}
\cite{BourgainBrezisMironescu2005}*{Open Problem 1}
\cite{mironescu2007sobolev}*{Theorem 3.2}
\cite{nguyen2008inequalities}*{Theorem 2}
\cite{Mironescu2008CRAS}*{Theorem 1}
\cite{Mironescu2008}\cite{Mironescu2010}
}.
By considering the energy 
\begin{multline*}
 \mathcal{G} (u, A) =\inf\bigl\{\mathcal{E}_{s,p}(\varphi_1,A)+\mathcal{E}_{1,sp}(\varphi_2,A)\;:\; u=\pi\compose (\varphi_1+\varphi_2)\\
 \varphi_1\in W^{s,p}(A,\R),\, \varphi_2\in W^{1,sp}(A,\R)\bigr\},
\end{multline*}
with \(\pi(t):=(\cos t,\sin t)\) for all \(t\in\R\), we recover the known linear estimates in this setting:

\begin{theorem}\label{weaklifting}
Let \(s\in (0,1]\), \(p\in [1,+\infty)\), \(m\in\N_*\), and let \(\mathcal{M}\) be a \(m\)-dimensional Riemannian manifold such that either \(\mathcal{M}\) is compact or \(\mathcal{M}=\R^m\). If for every map \(u\in W^{s,p}(\mathcal{M},\mathbb{S}^1)\) there exists a lifting \(\varphi\in W^{s,p}(\mathcal{M},\R)+W^{1,sp}(\mathcal{M},\R)\) such that \(u = \pi \compose\varphi\) almost everywhere in \(\mathcal{M}\), then there exists a constant \(C>0\) such that for every measurable function \(u:\mathbb{B}^m\to\mathbb{S}^1\), if either \(sp < m\) or \(\mathcal{E}^{s, p} (u, \mathbb{B}^m) \le 1/C\), there exist \(\varphi_1 \in W^{s,p}(\mathbb{B}^m,\R)\) and \(\varphi_2 \in W^{1,sp}(\mathbb{B}^m,\R)\) such that 
\(u = \pi \compose (\varphi_1 + \varphi_2)\) and 
\[
\mathcal{E}_{s,p}(\varphi_1,\mathbb{B}^m)+\mathcal{E}_{1,sp}(\varphi_2,\mathbb{B}^m)\le C\mathcal{E}_{s,p}(u,\mathbb{B}^m).
\]
\end{theorem}

\subsection{Superposition operators}
The proof of \cref{superposition} is obtained by considering the energy \(\mathcal{G} (u, A) = \mathcal{E}_{r, q } (f \compose u)\) in \cref{theoremCounterexample}.
When \(sp \le m\) and \(s < m\), it is possible to prove the uniform bound on the assumption that the superposition operator acts on a nonempty open set.

\begin{theorem}[Acting condition]\label{actingCondition}
Let \(s,r\in (0,1]\), \(p,q\in [1,+\infty)\) and \(m\in\N_*\) with \(rq=sp\), let \(\mathcal{M}\) be an \(m\)-dimensional Riemannian manifold which is either \(\R^m\) or compact, \(\mathcal{N}\) be a connected Riemannian manifold which is compact if \(sp > m\) or if \(s = p = m=1\), \(\mathcal{F}\) be a Riemannian manifold and let \(f:\mathcal{N}\to\mathcal{F}\) be a Borel-measurable map. If for every \(u \in W^{s,p} (\mathcal{M}, \mathcal{N})\), \(f \compose u \in W^{r, q} (\mathcal{M}, \mathcal{F})\),
then  there exists a constant \(C \in [0, + \infty)\), such that for every \(x, y \in \mathcal{N}\), if either \(sp < m\) or \( d_{\mathcal{N}} (x, y) \le 1/C\), then 
 \begin{equation}
 \label{itHolder}
  d_{\mathcal{F}} \big(f (x), f (y)\big) \le C d_{\mathcal{N}} (x, y)^{p/q}.
 \end{equation}
Moreover, when \(sp \le m\) and \(s < m\), the above statements are satisfied if there is a nonempty open set \(\mathcal{U} \subset W^{s, p} (\mathcal{M}, \mathcal{N})\) such that for every \(u \in \mathcal{U}\) one has \(f \compose u \in W^{r, q} (\mathcal{M}, \mathcal{F})\).
\end{theorem}
In particular, if the superposition operator given by \(f\) maps \(W^{s,p} (\mathcal{M}, \mathcal{N})\) into \(W^{r,q}(\mathcal{M}, \mathcal{F})\) and if \(p>q\), then \(f\) is constant on \(\mathcal{N}\).

When \(sp\ge m\), if \(p=q\) or if \(\mathcal{N}\) is compact (which is our assumption when \(sp>m\)), it is easy to see that one can avoid the smallness condition on \(d_\mathcal{N}(x,y)\) in the conclusion of \Cref{actingCondition}. 

When \(sp = m\) and \(\mathcal{N}\) is not compact, the H\"older continuity condition of \cref{actingCondition} implies that for every \(x, y \in \mathcal{N}\), 
\begin{equation}
\label{itHolderCritical}
 d_{\mathcal{F}} (f (x), f (y)) 
 \le C \bigl(d_{\mathcal{N}} (x, y)^{p/q} + d_{\mathcal{N}} (x, y)\bigr).
\end{equation}
If \(p=q\), the H\"older continuity condition of \cref{actingCondition} implies that \(f\) is Lipschitz-continuous; this condition is well-known to be necessary \citelist{\cite{MarcusMizel1979}\cite{Igari1965}\cite{Bourdaud1993}\cite{BourdaudSickel2011}\cite{Allaoui2009}}.

When \(s < 1\), the condition \eqref{itHolder} can be observed to be sufficient by a direct computation with the Gagliardo energy and relying, when \(sp = m\), on \eqref{itHolderCritical} and the fractional Gagliardo--Nirenberg interpolation inequality.

When \(r < s = 1\), the exact characterization of the superposition operators acting from the space \(W^{1,p} (\mathcal{M}, \mathcal{N})\) to  \(W^{r, q} (\mathcal{M}, \mathcal{F})\) remains open; when \(\mathcal{N} = \mathcal{F} = \R\) and \(f (t) = \abs{t}^{p/q}\), it is known that \(f\) maps \(W^{1,p} (\mathcal{M}, \R)\) to  \(W^{r, q} (\mathcal{M}, \R)\) \cite{Mironescu2015}.

\begin{proof}[Proof of \cref{actingCondition}]
\resetconstant
By \cref{superposition}, there exists a constant \(\Cl{supabstthm} >0\) such that for every measurable function 
\(u:\mathbb{B}^m\to\mathcal{N}\), 
\[
\mathcal{E}_{r,q}(f\compose u,\mathbb{B}^m)\le \Cr{supabstthm}\mathcal{E}_{s,p}(u,\mathbb{B}^m).
\]
We fix two points \(a_{\pm} = (\pm \frac{1}{2}, 0, \dotsc, 0)\) and we choose a function \(w \in C^\infty_c (\mathbb{B}^m,[-1,1])\) such that \(w = \pm 1\) on \(\mathbb{B}^m_{1/4} (a_\pm)\).
For \(x, y \in \mathcal{N}\), we consider a Lipschitz-continuous curve \(\gamma_{x, y} :[-1,1]\to\mathcal{N}\) satisfying \(\gamma_{x, y}(-1)=x\), \(\gamma_{x, y} (1) = y\) and \(\mathrm{Lip} \, \gamma_{x, y} \le d_{\mathcal{N}} (x, y)\).
Such a curve exists since \(\mathcal{N}\) is path--connected and a continuous path can always be reparametrized by arc--length. Since \(\gamma_{x, y}\) is Lipschitz-continuous, we have 
\[
 \mathcal{E}_{s, p} (\gamma_{x, y} \compose w, \mathbb{B}^m)
 \le (\mathrm{Lip} \,\gamma_{x, y})^{p} \mathcal{E}_{s, p} (w, \mathbb{B}^m)
 \le \mathcal{E}_{s, p} (w, \mathbb{B}^m) d_{\mathcal{N}} (x, y)^{p}.
\]
Next, we observe that \(f \compose \gamma_{x, y} \compose w = f(x)\)
on \(\mathbb{B}^m_{1/4} (a_-)\) and \(f \compose \gamma_{x, y} \compose w =f( y)\)
on \(\mathbb{B}^m_{1/4} (a_+)\). Therefore, we have when \(r \in (0, 1)\),
\[
\begin{split}
 \mathcal{E}_{r,q} (f \compose \gamma_{x, y} \compose w, \mathbb{B}^m)
 &\ge 2 \int_{\mathbb{B}^m_{1/4} (a_+)} \int_{\mathbb{B}^m_{1/4} (a_-)} 
  \frac{d_{\mathcal{F}} \big(f (x), f (y)\big)^q}{\abs{t - v}^{m + rq}} \diff t \diff v\\
  &\ge 2^{m + 1 + rq} \mathcal{L}^m (\mathbb{B}^m_{1/4})^2d_{\mathcal{F}} \big(f (x), f (y)\big)^q.
\end{split}
\]
When \(r = 1\), we have by H\"older's inequality,
\[
 \mathcal{E}^{1,q} (f \compose \gamma_{x, y} \compose w, \mathbb{B}^m)
 \ge \frac{\mathcal{E}^{1,1} (f \compose \gamma_{x, y} \compose w, \mathbb{B}^m)^q}{\mathcal{L}^m (\mathbb{B}^m)^{q - 1}}
\]
and 
\[
 \mathcal{E}^{1,1} (f \compose \gamma_{x, y} \compose w, \mathbb{B}^m)
 \ge \int_{[-\frac 12, \frac 12]\times \mathbb{B}^{m - 1}_{1/4}} \abs{D (f \compose \gamma_{x, y} \compose w)}\diff x
 \ge  \mathcal{L}^{m-1}(\mathbb{B}^{m - 1}_{1/4})d_{\mathcal{F}} \big(f (x), f (y)\big).
\]
The assertion \eqref{itHolder} then follows from the previous inequalities. 

The last statement follows from \cref{theoremDenseCounterexamples}.
\end{proof}

\section{The limiting case \(s = 0\)}
\label{case_s_0}
We consider the question about what the uniform boundedness becomes in the limit case \(s = 0\).
Looking at the proof \cref{theoremCounterexample}, it appears that the clustering step requires the condition \(s > 0\) to increase the energy. In order to bypass this difficulty, we assume that we have maps \(u : \mathbb{B}^m \to \mathcal{N}\) with a large Lebesgue energy \(\int_{\mathbb{B}^m} \abs{u}^p\).

\begin{theorem}
\label{theoremCounterexampleLp}
Let \(m,\, N \in \N_*\), \(p \in [1, +\infty)\) and let \(\mathcal{G}\) be an  energy over $\R^m$ with state space $\R^N$. Assume that for every measurable map \(u:\R^m\to\R^N\)
\begin{enumerate}[(i)]
\item (superadditivity) if the sets \(A,B\subset \R^m\) are open and if \(\Bar{A} \cap \Bar{B} = \emptyset\), then 
\[
\mathcal{G}\,(u,A\cup B)\geq \mathcal{G}\,(u,A)+\mathcal{G}\,(u,B),
\]
\item (scaling) for all $\lambda>0$, $h\in\R^m$ and any open set $A\subset\R^m$, one has 
\[
\mathcal{G}\,(u ,h+\lambda A)=\lambda^{m} \mathcal{G}\,(u(h + \lambda \cdot),A).
\]
\end{enumerate}
If for every \(u\in L^p(\mathbb{B}^m,\R^N)\), \(\mathcal{G}\,(u,\mathbb{B}^m) < + \infty\), then there exists \(C \in [0, +\infty)\) such that for every $u\in L^p(\mathbb{B}^m , \R^N)$, one has
\[
\mathcal{G}\,(u,\mathbb{B}^m)\le C \,\biggl(1 + \int_{\mathbb{B}^m}|u|^p \biggr).
\]
\end{theorem}

\Cref{theoremCounterexampleLp} allows one to recover classical results on \emph{superposition operators in Lebesgue spaces}. 
Given a Borel-measurable function \(f : \R^N \to \R^\ell\) and for every open set \(A \subset \R^m\) and every measurable function \(u:A\to\R^N\),
we set \(\mathcal{G}\,(u, A) = \int_{A} \abs{f \compose u}^p\).
By \cref{theoremCounterexampleLp}, if for every \(u \in L^p (\mathbb{B}^m,\R^N)\),  we have \(f \compose u \in L^p (\mathbb{B}^m,\R^\ell)\), then there exists a constant \(C \in [0, + \infty)\) such that for every \(u \in L^p (\mathbb{B}^m,\R^N)\) the following \emph{uniform bound} holds:
\[
 \int_{\mathbb{B}^m} \abs{f \compose u}^p \le C \Bigl( 1 + \int_{\mathbb{B}^m} \abs{u}^p\Bigr).
\]
By taking \(u\) to be a constant function, this implies in turn that for every \(t \in \R^N\),
\[
 \abs{f (t)} \le C' \bigl(1 + \abs{t}\bigr),
\]
which is a classical necessary and sufficient condition to have a superposition operator acting from \(L^p(\mathbb{B}^m,\R^N)\) to \(L^p (\mathbb{B}^m,\R^\ell)\) \cite{Krasnoselskii1964}*{Theorem 2.3} (see also \cite{AppellZabreiko1990}*{Theorem 3.1}).

\begin{proof}[Proof of \cref{theoremCounterexampleLp}]
\resetconstant
We assume by contradiction that there exists a sequence \((v_n)_{n \in \N}\) of measurable maps from \(\mathbb{B}^m \) to \(\R^N\) such that for each \(n \in \N\), we have \(\mathcal{G}\,(v_n,\mathbb{B}^m)<+\infty\), and such that
\[
\lim_{n \to \infty}
\frac{\mathcal{G}\,(v_n, \mathbb{B}^m)}{1+\displaystyle \int_{\mathbb{B}^m}\abs{v_n}^p}  = + \infty;
\]
we are going to construct a function \(u\in L^p(\mathbb{B}^m,\R^N)\) such that \(\mathcal{G}\,(u, \mathbb{B}^m) = + \infty\).

By rescaling \(v_n\) if \(\int_{\mathbb{B}^m}|v_n|>1\) and passing to a subsequence if necessary, we can assume that for each \(n \in \N\), there exists a function $u_n\in L^p( \mathbb{B}^m , \R^N)$ such that 
\begin{align*}
 \int_{\mathbb{B}^m}\abs{u_n}^p&\le 1 &
 & \text{ and } &
 2^{nm} &\le \mathcal{G}\,(u_n,\mathbb{B}^m) < + \infty.
\end{align*}
If \(Q\) denotes a cube of side-lenght \(1/\sqrt{m}\) contained in \(\mathbb{B}^m\), by dyadic decomposition, this cube \(Q\) contains a family of cubes of sidelengths \((2^{-n-1}/\sqrt{m})_{n \in \N}\) and thus, if we set \(\rho_n = 2^{-n - 2}/\sqrt{m}\), there exists a sequence of points \((a_n)_{n \in \N}\) such that the balls \(\bigl(\Bar{B}_{\rho_n} (a_n)\bigr)_{n \in \N}\) are disjoint balls contained in the open ball \(\mathbb{B}^m\). We define the map \(u : \mathbb{B}^m \to \R^N\) for each point \(x \in \mathbb{B}^m\) by 
\[
 u (x) = 
 \begin{cases}
   u_n \bigl(\frac{x - a_n}{\rho_n}\bigr) & \text{if \(x \in \mathbb{B}^m_{\rho_n} (a_n)\)},\\
   0 & \text{otherwise}.
 \end{cases}
\]
We have by countable superadditivity, translation-invariance and scaling of the energy \(\mathcal{G}\),
\[
 \mathcal{G}\,(u, \mathbb{B}^m) \ge \sum_{n \in \N} \mathcal{G}\,\bigl(u, \mathbb{B}^m_{\rho_n} (a_n)\bigr)
 = \sum_{n \in \N} \rho_n^{m} \mathcal{G}\,\bigl(u_n, \mathbb{B}^m\bigr)
 \ge \sum_{n \in \N} \Big(\frac{2^{-n - 2}}{\sqrt{m}}\Big)^m 2^{nm} = + \infty.
\]
On the other hand, we have
\[
 \int_{\mathbb{B}^m}|u|^p
 = \sum_{n \in \N} \rho_n^{m} \int_{\mathbb{B}^m}\abs{u_n}^p
 \le \sum_{n \in \N}\Big(\frac{2^{-n - 2}}{\sqrt{m}}\Big)^m < + \infty,
\]
thus ending the proof.
\end{proof}


\section{Higher order spaces}
\label{higher}

If $\mathcal{N}$ is a connected Riemannian manifold embedded in a Euclidean space $\R^\nu$ by a smooth embedding, and if \(\mathcal{M}\) is \(m\)-dimensional Riemannian manifold which is either Euclidean or compact, 
the nonlinear Sobolev space $W^{s,p}(\mathcal{M},\mathcal{N})$ can be defined extrinsically by 
\[
W^{s,p}(\mathcal{M},\mathcal{N})=\left\{u\in W^{s,p}(\mathcal{M},\R^\nu)\st u(x)\in\mathcal{N}\text{ for almost every }x\in\mathcal{M}\right\},
\]
where \(W^{s,p}(\mathcal{M},\R^\nu)\) is the usual linear higher order Sobolev space, that is the space of measurable maps \(u:\mathcal{M}\to\R^\nu\) such that \(\mathcal{E}_{s,p}(u,\mathcal{M})<+\infty\).

Here, if $s \in \N$ is an integer, the homogeneous Sobolev energy  $\mathcal{E}_{s,p}$ is defined for every measurable map \(u:\mathcal{M}\to\R^\nu\) by
\[
\mathcal{E}_{s,p}(u,\mathcal{M})=
\begin{dcases}
\int_\mathcal{M} |D^s u|^p &\text{if the \(s^{th}\)-order weak derivative \(D^s u\) belongs to \(L^p\)},\\
+\infty&\text{otherwise,}
\end{dcases}
\]
where $D^s u$ is understood as a $s$-linear map on $\R^m$ valued in $\R^\nu$, and $|\cdot|$ is any norm on the linear space composed by $s$-linear maps.
If $s \not \in \N$ is not an integer, we set
\[
\mathcal{E}_{s,p}(u,\mathcal{M})=
\begin{dcases}
\mathcal{E}_{s - \floor{s}, p}(D^{\floor{s}} u,\mathcal{M})&\text{if }u\in W^{\floor{s},p}(\mathcal{M},\R^\nu),\\
+\infty&\text{otherwise,}
\end{dcases}
\]
where $\mathcal{E}_{s - \floor{s},p}$ with \(s - \floor{s} \in (0, 1)\) has been defined in \eqref{eqDefGagliardo}
and  $D^{\floor{s}} u$ is a function from $\mathcal{M}$ valued in the normed linear space composed of $\floor{s}$-linear maps.

A generalization of \cref{uniform} is the following

\begin{theorem}[Higher order nonlinear uniform boundedness principle]
\label{uniformHigher}
Let \(m \in \N_*\), \(s \in (1, +\infty)\), \(p \in [1, +\infty)\), \(\mathcal{N}\) be a connected Riemannian manifold, which if \(sp > m\) or \(s = m = \frac{m}{p}\) is compact, and let \(\mathcal{G}\) be an  energy over $\R^m$ with state space $\mathcal{N}$. 
Assume that for every measurable map \(u:\R^m\to\mathcal{N}\)
\begin{enumerate}[(i)]
\item (superadditivity) for all open sets \(A,B\subset\R^m\) with disjoint closure,
\[
\mathcal{G}\,(u,A\cup B)\geq \mathcal{G}\,(u,A)+\mathcal{G}\,(u,B),
\]
\item (scaling) for all $\lambda>0$, $h\in\R^m$ and any open set $A\subset\R^m$,
\[
\mathcal{G}\,(u ,h+\lambda A)=\lambda^{m-sp} \mathcal{G}\,(u(h + \lambda \cdot),A).
\]
\end{enumerate}
If for every measurable function \(u : \mathbb{B}^m_2 \to \mathcal{N}\), \(\mathcal{E}_{s, p} (u, \mathbb{B}^m_2) < + \infty\) implies \(\mathcal{G}\,(u,\mathbb{B}^m) < + \infty\) and \(\mathcal{E}_{s, p} (u, \mathbb{B}^m_2) = 0\) implies \(\mathcal{G}\,(u,\mathbb{B}^m) = 0\), then there exists a constant $C \in [0, +\infty)$ such that for every measurable map $u:\mathbb{B}^m_2 \to \mathcal{N}$, if either \(sp < m\) or \(\mathcal{E}_{s, p} (u) \le 1/C\), then 
\begin{equation}
\label{conclusionHIGHER}
\mathcal{G}\,(u,\mathbb{B}^m)\leq C \left(\mathcal{E}_{s,p}(u,\mathbb{B}^m_2)+\mathcal{E}_{1,p}(u,\mathbb{B}^m_2)\right).
\end{equation}
\end{theorem}

Compared to \cref{uniform}, the conclusion of \cref{uniformHigher} has two weaknesses: the right-hand side contains a lower order energy \(\mathcal{E}_{1, p} (u, \mathbb{B}^m_2)\) and the energies in the right-hand side are evaluated on a larger ball than on the left-hand side. There are several ways to mitigate this issue.

\begin{rmk}
\Cref{uniformHigher} implies that if \(u : \mathbb{B}^m_1 \to \mathcal{N}\) is constant in the annulus \(\mathbb{B}^m_1 \setminus \mathbb{B}^m_{1/2}\),
then 
\[
\mathcal{G}\,(u,\mathbb{B}^m)\leq C \mathcal{E}_{s,p}(u,\mathbb{B}^m_1).
\]
Indeed, if we consider the extension \(\Bar{u} : \mathbb{B}^m_2 \to \mathcal{N}\) of \(u\) by the same constant, we have by a direct computation and by the Poincar\'e inequality,
\[
   \mathcal{E}_{s,p}(\Bar{u},\mathbb{B}^m_2)+\mathcal{E}_{1,p}(\Bar{u},\mathbb{B}^m_2)
   \le \mathcal{E}_{s,p}(u,\mathbb{B}^m_1).
\]
\end{rmk}

\begin{rmk}\resetconstant
When \(s < 1 + \frac{1}{p}\), we can conclude that 
\[
 \mathcal{G}\,(u,\mathbb{B}^m)\leq C \left(\mathcal{E}_{s,p}(u,\mathbb{B}^m_1)+\mathcal{E}_{1,p}(u,\mathbb{B}^m_1)\right).
\]
Indeed following the proof of \cref{extensionLemma}, we use the construction by Euclidean inversion of \(v:\mathbb{B}^m_\lambda\to\mathcal{N}\) by \eqref{eqDefvInversion}. We have then 
\[
\begin{split}
\int_{\mathbb{B}^m_\lambda\setminus\mathbb{B}^m}\int_{\mathbb{B}^m}\frac{|Dv(y) -D v(x)|^p\diff x\diff y}{|x-y|^{m+(s-1) p}}
& \le 2^{p - 1} \int_{\mathbb{B}^m_\lambda\setminus\mathbb{B}^m}\int_{\mathbb{B}^m}\frac{|Dv(y)|^p  + |D v(x)|^p\diff x\diff y}{|x-y|^{m+(s-1) p}}\\
& \le 2^{p - 1} \int_{\mathbb{B}^m_\lambda\setminus\mathbb{B}^m} |Dv(y)|^p\biggl(\int_{\mathbb{B}^m}\frac{\diff x}{|x-y|^{m+(s-1) p}}\biggr) \diff y\\
&\qquad 
+\int_{\mathbb{B}^m} |Dv(x)|^p\biggl(\int_{\mathbb{B}^m_\lambda\setminus\mathbb{B}^m}\frac{\diff y}{|x-y|^{m+(s-1) p}}\biggr) \diff x\\
&\le \C \int_{\mathbb{B}^m}\frac{|Du(x)|^p\diff x}{(1-|x|)^{(s-1)p}}.
\end{split}
\]
By the Hardy inequality for fractional Sobolev spaces \cite{Dyda2004}*{(17)}, we have 
\[
  \int_{\mathbb{B}^m}\frac{|Du(x)|^p\diff x}{(1-|x|)^{(s-1)p}}
  \le \C \left(\mathcal{E}_{s,p}(u,\mathbb{B}^m_1)+\mathcal{E}_{1,p}(u,\mathbb{B}^m_1)\right).
\]
\end{rmk}

The proof of \cref{uniformHigher} is rather similar to that of \cref{uniform} (corresponding to the case \(s\leq 1\)). The main change is in the opening lemma and this is why we need the additional term \(\mathcal{E}_{1,p}(u,\mathbb{B}^m_2)\) in the estimate of \(\mathcal{G}\,(u,\mathbb{B}^m)\). Here, we will not give a detailed proof of \cref{uniformHigher}; we only state and prove an opening lemma for higher order Sobolev maps.

\begin{lem}\label{openinghigher}
Let \(m\in\N_*\), \(s\in (1,+\infty)\), \(p\in [1,+\infty)\) and \(\lambda>1\), \(\eta\in (0,\lambda)\). For every \(\varphi \in C^\infty (\mathbb{B}^m_{(1+\eta)\rho}, \mathbb{B}^m_{(\lambda-\eta)\rho})\), there exists a constant \(C>0\) such that for every \(\rho>0\) and every measurable map \(u:\mathbb{B}^m_{\lambda\rho}\to\mathcal{N}\), there exists a point \(a\in\mathbb{B}^m_{\eta\rho}\) such that
\begin{equation*}
\mathcal{E}_{s,p}(u\compose (\varphi(\cdot-a)+a),\mathbb{B}^m_\rho)\le C\big(\mathcal{E}_{1,p}(u,\mathbb{B}^m_{\lambda\rho})+\mathcal{E}_{s,p}(u,\mathbb{B}^m_{\lambda\rho})\big).
\end{equation*}
\end{lem}

The lower-order term in \cref{uniformHigher} comes from the estimate of \cref{openinghigher}.
This lower-order term cannot be removed: if \(u\) is linear and \(\varphi\) is not a polynomial of degree at most \(\lceil s\rceil-1\), where $\lceil s\rceil$ stands for the smallest integer greater than or equal to $s$, then \(\mathcal{E}_{s,p}(u,\mathbb{B}^m_{\lambda\rho}) = 0\) and for every \(a \in \mathbb{B}^m_{\eta \rho}\), \(\mathcal{E}_{s,p}(u\compose (\varphi(\cdot-a)+a),\mathbb{B}^m_\rho) > 0\).

\begin{proof}[Proof of \cref{openinghigher}]%
\resetconstant
We define for each \(a\in\mathbb{B}^m_{\eta\rho}\) the map \(\varphi_a=(\varphi(\cdot-a)+a):\mathbb{B}^m_\rho\to\mathbb{B}^m_{\lambda\rho}\). We will prove the average estimate
\begin{equation}\label{meanenergyhigher}
\fint_{\mathbb{B}^m_{\eta\rho/2}} \mathcal{E}_{s,p}(u \compose \varphi_a,\mathbb{B}^m_\rho)\diff a\leq C\big(\mathcal{E}_{1,p}(u,\mathbb{B}^m_{\lambda\rho})+\mathcal{E}_{s,p}(u,\mathbb{B}^m_{\lambda\rho})\big).
\end{equation}
\emph{In the case \(s\in\N_*\)}, we follow \cite{BousquetPonceVanSchaftingen2015}*{Lemma 2.3}; by an easy induction over $s$ and a rather classical approximation procedure, one gets the following claim:
\begin{claim}\label{chain}
For every $u\in W^{s,p}(\mathbb{B}^m_{\lambda\rho},\R^\nu)$,  $\varphi\in C^\infty (\mathbb{B}^m_{\rho},\mathbb{B}^m_{\lambda\rho})$, for almost every  $x\in\mathbb{B}^m_{\rho}$ and $h=(h_1,\dots,h_s)\in (\R^m)^s$,
\begin{equation}\label{chainrule}
D^s (u\compose\varphi)(x)[h]
=\sum_{k=1}^s\sum_{J\in\mathcal{P}_k(s)}\,c_{k,J} \,(D^ku)(\varphi(x))\bigl[D^{|J_1|}\varphi(x)[h_{J_1}],\dotsc,D^{|J_k|}\varphi(x)[h_{J_k}]\bigr],
\end{equation}
where $\mathcal{P}_k(s)$ is the set of all partitions $J=(J_1,\dots,J_k)$ of $\{1,\dots,s\}$ in $k$ non empty sets, the $c_{k,J}\in\R$ are some constants depending on $k,J$, and $h_J:=(h_j)_{j\in J}$ for every non empty subset $J\subset\{1,\dots,s\}$.
\end{claim}
As a consequence of \cref{chain}, for almost every $x\in \mathbb{B}^m_\rho$, we get the estimate
\[
|D^s(u\compose\varphi)(x)|
\leq \sum_{\substack{k, j_1, \dots, j_k\in \{1,\dots,s\}\\
j_1+\dots+j_k=s}} |c_{k,J}|\, |(D^ku)(\varphi(x))|\, |D^{j_1}\varphi(x)|\dotsm |D^{j_k}\varphi(x)|.
\]
By Young's inequality for products, we have 
\[
  |D^{j_1}\varphi(x)|\dotsm |D^{j_k}\varphi(x)|
  \le \frac{j_1}{s} |D^{j_1}\varphi(x)|^\frac{s}{j_1} + \dotsb  + \frac{j_k}{s} |D^{j_k}\varphi(x)|^\frac{s}{j_k}
\]
and thus
\[
|D^s(u\compose\varphi)(x)|\leq \C\sum_{k=1}^s \sum_{l = 1}^{s-k +1}|(D^ku)(\varphi(x))|\, |D^l\varphi|^{\frac sl}.
\]
Applying the inequality to our functions $u$ and $\varphi_a$, and integrating the inequality over $x\in\mathbb{B}^m_\rho$ and $a\in\mathbb{B}^m_{\eta\rho/2}$ yield a constant $\Cl{lpm}$ (depending on $\varphi$) such that
\begin{align*}
\int_{\mathbb{B}^m_{\eta\rho/2}}\mathcal{E}_{s,p}(u \compose \varphi_a,\mathbb{B}^m_\rho)\diff a&=\int_{\mathbb{B}^m_{\eta\rho/2}}\int_{\mathbb{B}^m_\rho}|D^s (u \compose \varphi_a)|^p \diff x\diff a\\
&\leq \Cr{lpm} \sum_{k = 1}^s \int_{\mathbb{B}^m_{\eta\rho/2}}\int_{\mathbb{B}^m_\rho}|D^k u|^p(a+\varphi(x-a))\diff x\diff a.
\end{align*}
By changes of variable \(y = x - a\in \mathbb{B}^m_{(1+\frac \eta 2)\rho}\) and \(w=a+\varphi(y)\in \mathbb{B}^m_{(\lambda-\frac\eta 2)\rho}\), we are led to the estimates
\begin{equation}
\label{fubiniblue}
\int_{\mathbb{B}^m_{\eta\rho/2}}\mathcal{E}_{s,p}(u \compose \varphi_a,\mathbb{B}^m_\rho)\diff a\leq \C\sum_{k=1}^{s}\int_{\mathbb{B}_{(\lambda-\frac\eta 2)\rho}}|D^ku|^p(w)\diff w\leq \C\big(\mathcal{E}_{1,p}(u,\mathbb{B}^m_{\lambda\rho})+\mathcal{E}_{s,p}(u,\mathbb{B}^m_{\lambda\rho})\big).
\end{equation}
This ends the proof in the integer case.

\medbreak

\noindent
\emph{When \(s\) is not an integer}, we  estimate the norm of the difference $D^{\floor{s}}(u \compose \varphi_a)(x)-D^{\floor{s}}(u \compose\varphi_a)(y)$. 
We use \cref{chain} in order to express $D^{\floor{s}}(u \compose \varphi_a)(x)$ as a linear combination of expressions of the form $F(x):=L(x)[H_1(x),\dots,H_k(x)]$ with $L(x)=(D^ku)(\varphi_a(x))$ and $H_l(x)[h]=D^{|J_l|}(\varphi_a)(x)[h_{J_l}]$. 
We recall that if $\Phi$ is a multilinear map defined on a cross product of linear spaces $A_1\times\dots\times A_k$ and $a=(a_1,\dots,a_k),\, b=(b_1,\dots,b_k)\in A_1\times \dots\times A_k$ then
\[
\Phi(a)-\Phi(b)=\sum_{l=1}^k \Phi(b_1,\dots,b_{l-1},a_l-b_l,a_{l+1},\dots,a_k).
\]
In particular, we have the estimate
\[
|F(x)-F(y)|\leq |L(x)-L(y)|\prod_{l=1}^k |H_l(x)|+\sum_{m=1}^k\bigg( |L(y)|\, |H_m(x)-H_m(y)|\prod_{l\neq m} |H_l(x)|\bigg).
\]
Since $\varphi$ is smooth, each map $H_l$ is smooth thus yielding a constant $\Cl{gzed}>0$ depending on $\varphi$ such that
\begin{multline*}
|D^{\floor{s}}(u \compose \varphi_a)(x)-D^{\floor{s}}(u \compose \varphi_a)(y)|\\
\le \Cr{gzed}\bigg(\sum_{k=1}^{\floor{s}} |(D^ku)(\varphi_a(x))-(D^ku)(\varphi_a(y))|+|(D^ku)(\varphi_a(y))|\,\abs{x - y}\bigg).
\end{multline*}
Thus, by integration we get
\begin{multline}\label{inter}
\mathcal{E}_{s - \floor{s},p}(D^{\floor{s}}(u \compose \varphi_a),\mathbb{B}^m_\rho)\\
\leq \Cl{plo}\sum_{k=1}^{\floor{s}}\mathcal{E}_{s - \floor{s},p}(D^k u\compose\varphi_a,\mathbb{B}^m_\rho)
+
\Cr{plo}\sum_{k=1}^{\floor{s}}
\int_{\mathbb{B}^m_\rho\times\mathbb{B}^m_\rho} \frac{|(D^ku)(\varphi_a(y))|^p\diff x\diff y}{\abs{x - y}^{m+(s - \floor{s}-1)p}}.
\end{multline}
Since $m+(s - \floor{s}-1)p<m$, the second term in the right-hand side of \eqref{inter} is lower or equal than
\[
\Cr{plo}
\sum_{k=1}^{\floor{s}}
\int_{\mathbb{B}^m_\rho}|(D^ku)(\varphi_a(x))|^p\diff x\ \int_{\mathbb{B}^m_{2\rho}} \frac{\diff y}{|y|^{m+(s - \floor{s}-1)p}}
\leq 
\C \sum_{k=1}^{\floor{s}}\int_{\mathbb{B}^m_\rho}|(D^ku)(\varphi_a(x))|^p\diff x ,
\]
whose average over $\mathbb{B}^m_{\eta\rho/2}$ is controlled by \(\mathcal{E}_{1,p}(u,\mathbb{B}^m_{\lambda\rho})+\mathcal{E}_{s,p}(u,\mathbb{B}^m_{\lambda\rho})\), by the same changes of variables that that leading to \eqref{fubiniblue}. 

Moreover, by the proof of Lemma \ref{openingLemma}, the average over $\mathbb{B}^m_{\eta\rho/2}$ of the first term in the right-hand side of \eqref{inter} is lower or equal than
\begin{align*}
 \C\sum_{k=1}^{\floor{s}}\mathcal{E}_{s - \floor{s},p}(D^ku,\mathbb{B}^m_{\lambda\rho})\leq \C\big(\mathcal{E}_{1,p}(u,\mathbb{B}^m_{\lambda\rho})+\mathcal{E}_{s,p}(u,\mathbb{B}^m_{\lambda\rho})\big),
\end{align*}
thus ending the proof.
\end{proof}

\begin{rmk}
In the proof of \cref{uniformHigher} outlined above, it appears that the Sobolev maps are only precomposed. 
This implies that all the pointwise estimate in the proofs for when \(s \in \N_*\) are still valid for intrinsic covariant derivatives \cite{ConventVanSchafingen2017} and thus \cref{uniformHigher} holds for intrinsic weak covariant derivatives.
\end{rmk}

\end{document}